\pdfoutput=1
\documentclass[final]{amsart}
\usepackage{pdfsync}
\usepackage{amsmath,amssymb}
\usepackage{enumerate}
\usepackage{xspace}
\usepackage{boxedminipage}
\usepackage{algorithm}
\usepackage{algpseudocode}
\usepackage{listings}% for including source code
\usepackage{tabularx}
\usepackage{subcaption}
\usepackage{Aaron_Pim_Style}
\usepackage{tristan}
\usepackage{pgfplots}
\usepackage{pgfplotstable}
\usepackage{fullpage}
\usepackage[backend=biber, style=numeric, giveninits, maxbibnames=9]{biblatex} 
\usepackage{pgfplots}
\usepackage{pgffor} % For the foreach loop
\pgfplotsset{compat=newest}
\usepgfplotslibrary{groupplots} % Include the groupplots library

\newcounter{custom}
\numberwithin{custom}{section} % Ensure it resets with each new section

\let\oldsubsection\subsection
\renewcommand{\subsection}{\stepcounter{custom}\oldsubsection}

\newtheorem{theorem}[custom]{Theorem}

\newtheorem{lemma}[custom]{Lemma}

\newtheorem{remark}[custom]{Remark}
\newtheorem{definition}[custom]{Definition}

\renewcommand{\vec}[1]{\geovec{#1}}
\renewcommand{\mat}[1]{\geomat{#1}}

\newcommand{\Hspace}{\ensuremath{\operatorname{H}}}

\renewcommand{\avg}[1]{\{\kern-1.2mm\{#1\}\kern-1.2mm\}}

\numberwithin{equation}{section}
\setboolean{showtodo}{true}
\setboolean{shownotes}{true}
\setboolean{showchanges}{false}%true
\setboolean{usemathrsfs}{false}%true

\newboolean{showTikz}
\setboolean{showTikz}{false}

\usepackage{tikz}
\usetikzlibrary{calc}
\usepackage{pgfplots}
\pgfplotsset{width=7cm,compat=1.14}
\usepgfplotslibrary{colorbrewer}

\author[Charalambos G. Makridakis]{Charalambos G. Makridakis}
\address{DMAM, University of Crete / 
Institute of Applied and Computational Mathematics, 
FORTH, 700$\,$13 Heraklion, Crete, Greece, and MPS, University of  Sussex, Brighton BN1 9QH, United Kingdom} 
\email {\href{mailto:C.G.Makridakis@iacm.forth.gr}{C.G.Makridakis{\it @\,}iacm.forth.gr}}

\author{
  Aaron Pim
}
\address{
  Aaron Pim,
  \thanks{
    Department of Mathematical Sciences,
    University of Bath, Claverton Down, Bath BA2 7AY, UK.
    {\tt{A.R.Pim@bath.ac.uk}}
}}

\author{
  Tristan Pryer
}
\address{
  Tristan Pryer,
  \thanks{
    Department of Mathematical Sciences,
    University of Bath, Claverton Down, Bath BA2 7AY, UK.
    {\tt{tmp38@bath.ac.uk}}
}}

\title
{Deep Uzawa for PDE Constrained Optimisation}
\date{\today}
\pdfformat{true}
\begin{document}

\begin{abstract}
In this work, we present a numerical solver for optimal control
problems constrained by linear and semi-linear second-order elliptic
PDEs. The approach is based on recasting the problem and includes an
extension of Uzawa's algorithm to build approximating sequences for
these constrained optimal control problems. We prove strong
convergence of the iterative scheme in their respective norms, and
this convergence is generalised to a class of restricted function
spaces. We showcase the algorithm by demonstrating its use numerically
with neural network methods that we coin Deep Uzawa Algorithms and
show they perform favourably compared with some existing Deep Neural
Network approaches.
\end{abstract}

\maketitle

\section{Introduction}

Optimal control problems subject to physical constraints such as
partial differential equations (PDEs) represent an important class of
problems in both science and engineering. These problems have applications
in various domains, such as fluid flow \cite{DolgovStoll:2017}, heat
conduction \cite{HarbrechtKunothSimonciniUrban:2023}, structural
optimisation \cite{KolvenbachLassUlbrich:2018}, crystal growth
\cite{ChunHesthaven:2008} and radiotherapy
\cite{CoxHattamKyprianouPryer:2023},  to name just a few. The ability to
quickly and accurately model and solve these problems is important for
advancements in these fields.

The mathematical theory of distributed and boundary optimal control
problems subject to linear and nonlinear, elliptic and parabolic PDEs
is well established, providing a robust theoretical foundation for
addressing these challenges
\cite{BieglerGhattasHeinkenschlossBloemen-Waanders:2003,ReesDollarWathen:2010,De-los-Reyes:2015}. Numerically,
common approaches to solving these problems involve formulating
first-order optimality systems, which typically introduce either
adjoint states or Lagrange multipliers. These formulations help in
deriving the necessary conditions for optimality and form the basis
for various numerical algorithms.

The continuous formulations of these optimal control problems are
often discretised using finite element methods
\cite{AllendesFuicaOtarola:2020, BrennerGedickeSung:2018}, spectral methods
\cite{RodenMills-WilliamsPearsonGoddard:2022} or wavelet methods
\cite{Kunoth:2005}. Finite element methods are widely used due to
their flexibility in handling complex geometries and boundary
conditions, while spectral methods are preferred for their high
accuracy in solving problems with smooth solutions. Wavelet methods
are employed because they offer a multi-resolution analysis, which can
efficiently capture localised features and singularities in the
solution. All three methods transform the continuous problem into a
discrete one, which can then be tackled using established optimisation
algorithms such as gradient descent, semismooth Newton methods and the
alternating direction method of multipliers
\cite{CiarletMiaraThomas:1989, Bertsekas:1999}. The choice of discretisation method
often depends on the specific characteristics of the problem at hand,
such as the nature of the domain, the smoothness of the solution, and
the computational resources available.

In practical implementation, these discretised formulations are
iteratively solved to update the control variables, ensuring
convergence to an optimal solution. Despite the effectiveness of these
traditional numerical methods, they can be computationally intensive,
particularly for high-dimensional problems or those with complex
constraints. This has led to the exploration of novel approaches,
including the integration of deep neural networks, to improve
computational efficiency and accuracy in solving PDE-constrained
optimal control problems
\cite{PINNs_augmented_Lagrangian,Barry-StraumeSarsharPopovSandu:2022,SirignanoMacArtSpiliopoulos:2023,LuoOLeary-RoseberryChenGhattas:2023}. For
instance, physics-informed neural networks (PINNs)
\cite{PINNs_original_2019} have been adapted to enforce hard
constraints in challenging topology optimisation problems, showing
promise in handling the high dimensionality and PDE constraints
characteristic of such inverse design problems
\cite{PINNs_augmented_Lagrangian}. In
\cite{Barry-StraumeSarsharPopovSandu:2022,SirignanoMacArtSpiliopoulos:2023,LuoOLeary-RoseberryChenGhattas:2023},
two primary classes of methods have emerged: penalty-based methods and
adjoint-based methods. Penalty methods transform the PDE-constrained
optimisation problem into an unconstrained one by introducing an
augmented loss function that includes the PDE residual. However, large
penalty weights, necessary for enforcing PDE constraints accurately,
can make the optimisation problem ill-conditioned and the choice of
these weightings is not clear apriori. Adjoint-based methods on the
other hand solve the adjoint state equation to compute the derivative
of the Lagrangian with respect to the control variable, iteratively
updating the control.

The nature of the PDE constraint is also important. We illustrate our
approach for a classical linear constraint, but nonlinear constraints
are important in applications
\cite{CasasChrysafinos:2018,EndtmayerLangerNeitzelWickWollner:2020,PichiStrazzulloBallarinRozza:2022}. For
example rigorous understanding of semilinear Allen-Cahn type
constraints are challenging and the topic of ongoing research
\cite{ChrysafinosPlaka:2023}. We give some numerical experiments to
demonstrate our method is also applicable to a stationary Allen-Cahn
equation.

In this work, we extend Uzawa's algorithm \cite{Uzawa:1958} to develop
approximating sequences for PDE-constrained optimal control problems
using neural networks, introducing the Deep Uzawa Algorithms. Uzawa
methods are a well known class of algorithm that generate iterative
approximations for saddle point problems, \cite{Uzawa:1958,
  CiarletMiaraThomas:1989}. Standard convergence proofs
\cite{CiarletMiaraThomas:1989} impose severe restrictions on mesh
parameters when applied to PDE-constrained optimisation problems. The
algorithm performs better with augmented Lagrangians, but it is still
not the preferred method when using finite elements or other
traditional PDE discretisation approaches. This is likely due to the
appearance of high-order gradients in the optimality
conditions. However, neural networks offer an alternative to
discretisation with characteristics different from typical finite
elements, being naturally suited to minimise energy functionals.

Our approach is to frame PDE-constrained optimisation problems in a
functional analytic context where the Uzawa algorithm is fundamentally
applicable.  We design the resulting energies to inherently possess
strong coercivity properties.  Consequently, the resulting \emph {deep
Uzawa} neural network algorithms become highly efficient and robust. A
key feature of our approach is the direct updating of the adjoint
state within the iterative scheme, eliminating the need for an
additional network for the adjoint variable. This makes our method
computationally competitive with penalty-based approaches. As an
initial step towards comprehensive convergence analysis, we provide
rigorous mathematical proofs demonstrating the convergence of our
iterative scheme at the PDE level (assuming exact solutions for the
PDEs and energy minimisation problems) within their respective norms,
and we further generalise this convergence to restricted function
spaces. In forthcoming work, we will extend our analysis to account
for errors due to neural network approximations.

To validate the practical utility of our approach, we
present numerical examples that extend our method to semi-linear PDE
constraints. This work represents an advancement in the integration of
neural network architectures with traditional optimisation techniques,
offering a robust and efficient solution framework for complex optimal
control problems.

The remainder of the paper is structured as follows: In \S
\ref{sec:pdecon} we introduce the PDE constrained optimisation
framework, we explore appropriate modifications of the underlying
Lagrangian and in \S \ref{sec:conv} we state the main results of this
contribution. In \S \ref{sec:deepuzawa} we introduce the neural network
approximation and the Deep Uzawa method. In \S \ref{sec:PDECon_lin_constr} we illustrate various numerical experiments with linear constraints, extending these experiments to semi-linear constraints in \S \ref{sec:Allen-Cahn}. Finally, \S \ref{sec:proof} concludes the
paper with the formal proofs of the main theorems.

\section{Problem Setting and Approach}  
\label{sec:pdecon}

\subsection{Model PDE Constrained Optimisation}

Let $\W \subset \mathbb{R}^d$ be a bounded, convex domain. Throughout
this work we denote the standard Lebesgue spaces by $\leb p(\omega)$,
$1\le p\le \infty$, $\omega\subset\mathbb{R}^d$, $d=1,2,3$, with
corresponding norms $\|\cdot\|_{L^p(\omega)}$. In the case $p=2$, we use $\Norm{\cdot}$ and $\ip{\cdot}{\cdot}$ to denote the norm and inner product, respectively. We then introduce the
Sobolev spaces \cite{Evans}
\begin{equation}
  \sobh{k}(\W) 
  := 
  \ensemble{\phi\in\leb{2}(\W)}
  {\D^{\vec\alpha}\phi\in\leb{2}(\W), \text{ for } \norm{\geovec\alpha}\leq k},
\end{equation}
which are equipped with norms and semi-norms
\begin{gather}
  \Norm{u}_{\sobh{k}(\W)}^2 
  := 
  \sum_{\norm{\vec \alpha}\leq k}\Norm{\D^{\vec \alpha} u}_{\leb{2}(\W)}^2 
  \AND \norm{u}_{k}^2 
  :=
  \norm{u}_{\sobh{k}(\W)}^2 
  =
  \sum_{\norm{\vec \alpha} = k}\Norm{\D^{\vec \alpha} u}_{\leb{2}(\W)}^2
\end{gather}
respectively, where $\vec\alpha = \{ \alpha_1,...,\alpha_d\}$ is a
multi-index, $\norm{\vec\alpha} = \sum_{i=1}^d\alpha_i$ and
derivatives $\D^{\vec\alpha}$ are understood in a weak sense.

Optimal control problems constrained by PDEs involve finding a control
function that minimises a given cost functional while satisfying a PDE
constraint. To that end, let $\Hspace(\W) := \sobh{2}(\W) \cap \sobhz{1}(\W)$ and
\begin{equation}
  K : \Hspace(\W) \times \leb{2}(\W) \to \leb{2}(\W)
\end{equation}
denote the residual of a second order PDE.

For a cost functional $\widehat{J}: \Hspace(\W) \times \leb{2}(\W) \to
\reals$, the constrained minimisation problem is to find
\begin{equation}
  \label{eq:abstractPDECO}
  \begin{split}
    (u^*, f^*) = \argmin_{u, f \in \Hspace(\W) \times \leb{2}(\W)} \widehat{J}(u,f)
    \\
    \text{subject to } K(u,f) = 0 \text{ in }\W,
  \end{split}
\end{equation}
where $u$ is the state variable and $f$ is the control variable. Let
\begin{equation}
  \leb2_0(\W)
  :=
  \setbra{\phi \in \bra{\leb2 \cap \rm{BV}}(\W): \quad \phi(x) = 0, ~\forall x \in \partial \W}.
\end{equation}
One approach to solving the above problem is to introduce a
Lagrangian, $\widehat{L}: \Hspace(\W) \times \leb{2}(\W) \times
\leb2_0(\W) \to \reals$, where the constraint is imposed via a
Lagrange multiplier $z \in \leb2_0(\W)$. The Lagrangian is defined as
\begin{equation}
  \widehat{L}(u,f,z) := \widehat{J}(u,f) + \ltwop{K(u,f)}{z}.
\end{equation}

This gives rise to a saddle point problem whereby
(\ref{eq:abstractPDECO}) is equivalent to finding
\begin{equation}
  \label{eq:minmax}
  (u^*, f^*, z^*)
  =
  \argmin_{u,f \in \Hspace(\W) \times \leb{2}(\W)}
  \argmax_{z \in \leb2_0(\W)} \widehat{L}(u,f,z).
\end{equation}

The optimal solution can be determined through the first-order
optimality conditions, which are computed by examining where the
Fr\'echet derivative of the Lagrangian vanishes, that is,
\begin{equation}
  \label{eq:generalised_lagrangian_derviative}
  \D \widehat{L}(u,f,z) = 
  \begin{bmatrix}
    \widehat{L}_u (u,f,z)
    \\
    \widehat{L}_f (u,f,z)
    \\
    \widehat{L}_z(u,f,z)
  \end{bmatrix}
  =
  \begin{bmatrix}
    \widehat{J}_u(u, f) + \ltwop{K_u(u,f)}{z}
    \\
    \widehat{J}_f(u, f) + \ltwop{K_f(u,f)}{z}
    \\
    K(u,f)
  \end{bmatrix}
  = \vec{0}.
\end{equation}
Most traditional methods solve the above system of equations by
employing various PDE discretisation methods. It is apparent that it
is indeed computationally challenging to solve this fully coupled
system, especially in high-dimensional problems or when the PDE
constraint is time-dependent.

To illustrate the main idea behind our approach, we consider a simple
linear-quadratic PDE-constrained optimisation problem. Note however
that our approach can be extended to much more involved problems. Let
the cost functional $\widehat{J}$ be defined as
\begin{equation}
  \label{eq:costfunctional}
  \widehat{J}(u,f) 
  := 
  \frac{1}{2} \Norm{u-\mathcal{D}}^2_{\leb{2}(\W)}
  +
  \frac{\alpha}{2}\Norm{f}^2_{\leb{2}(\W)},
\end{equation}
where $\mathcal{D}$ is the target state, and $\alpha > 0$ is a
regularisation parameter. The PDE constraint is given by
\begin{equation}
  \label{eq:pdeconstraint}
  \begin{split}
    K(u,f) := \Delta u  + f &= 0 \text{ in } \W,
  \end{split}
\end{equation}
with zero Dirichlet boundary conditions.
The Lagrangian functional corresponding to this problem is
\begin{equation}
  \widehat{L}(u,f,z)
  =
  \frac{1}{2} \Norm{u-\mathcal{D}}^2_{\leb{2}(\W)}
  +
  \frac{\alpha}{2}\Norm{f}^2_{\leb{2}(\W)}  
  -
  \ltwop{K(u,f)}{z}.
\end{equation}
Computing the Fr\'echet derivatives of $\widehat{L}$ with respect to
$u$, $f$, and $z$, we obtain the system
\begin{equation}
  \label{eq:frechet_derivatives}
  \begin{split}
    - \Delta z + u &= \mathcal{D}
    \\
    \alpha f - z &= 0
    \\
    -\Delta u - f &= 0
  \end{split}
\end{equation}
This system represents the first-order optimality conditions for the
PDE-constrained optimisation problem. The optimal solution $(u^*, f^*,
z^*) \in \Hspace(\W) \times \leb{2}(\W) \times \leb{2}_0(\W)$ determines a
saddle point of the Lagrangian.

\subsection{A strongly coercive Lagrangian}

A key observation is that, although the optimality conditions
\eqref{eq:frechet_derivatives} is a well posed PDE system, rarely one
solves it using direct methods, see e.g. \cite{ReesDollarWathen:2010},
for the finite element case. As mentioned, a particular attractive
iterative method at the PDE level is Uzawa, which produces iterates,
 \begin{equation}\label{eq:Uzawa scheme Poisson_Lhat}
     \begin{split}
       (u^k, f^k) &= \argmin_{u, f \in \Hspace(\W) \times \leb{2}(\W)}
       \widehat{L}(u, f, z^k) \\ z^{k+1} &= z^k + \rho K(u^k, f^k),
     \end{split}
   \end{equation}
 see Section \ref{iter_PDE} for precise definitions.

 The primary advantage of \eqref{eq:Uzawa scheme Poisson_Lhat} is that
 each iteration involves solving a minimisation problem, and the
 Lagrange multiplier is updated straightforwardly. This formulation is
 particularly well-suited for neural network approximations. However,
 when finite elements are used, the algorithm's stability is known to
 be guaranteed only under strict mesh restrictions. Furthermore, for
 neural networks, methods based on \eqref{eq:Uzawa scheme
   Poisson_Lhat}, we have observed that the scheme exhibits
 computational instability; we refer to the next section for detailed
 formulations of neural network approximations.

To address this issue, we introduce the modified functional
\begin{equation}\label{eq:altcostfunctional1}
  J(u,f)
  :=
  \frac{1}{2} \Norm{u-\mathcal{D}}^2_{\leb{2}(\W)}
  +
  \frac{\alpha}{4}\Norm{f}^2_{\leb{2}(\W)}
  +
  \frac{\alpha}{4}\Norm{\Delta u}^2_{\leb{2}(\W)},
\end{equation}
and seek
\begin{equation}\label{eq:altcostfunctional}
  (u^*, f^*)
  =
  \argmin_{u, f \in \Hspace(\W) \times \leb{2}(\W)} J(u,f)
\end{equation}
subject to the constraint
\begin{equation}\label{eq:altcostfunctional3}
  K(u,f) = 0 \text{ in } \W.
\end{equation}

The modification to the cost functional $J(u,f)$ includes a
consistent regularisation term $\frac{\alpha}{4}\Norm{\Delta
  u}^2_{\leb{2}(\W)}$, which helps to balance the relative strength of
the functional with the regularity requirements imposed by the PDE
constraint. This regularisation term is important in ensuring the
well-posedness of the optimisation problem and improving the
convergence properties of the iterative solution methods.

The solution to the modified problem is equivalent to the solution to
the minimisation problem, given in equation \eqref{eq:costfunctional},
which may be seen by comparing the first-order optimality conditions of each
problem. Indeed, let us define
\begin{equation}\label{eq:Lagrangian}
  L(u,f,z) := J(u,f) + \ltwop{K(u,f)}{z}.
\end{equation}
The first-order optimality conditions are 
\begin{equation}
  \D L(u,f,z) =
  \begin{bmatrix}
    L_u (u,f,z)
    \\
    L_f (u,f,z)
    \\
    L_z(u,f,z)
  \end{bmatrix}
  =
  \begin{bmatrix}
    J_u(u, f) + \ltwop{K_u(u,f)}{z}
    \\
    J_f(u, f) + \ltwop{K_f(u,f)}{z}
    \\
    K(u,f)
  \end{bmatrix}
  = \vec{0}.
\end{equation}
These conditions yield
\begin{equation}
  \label{eq:kkt-uzawa}
  \begin{split}
    \frac{\alpha}{2} \Delta^2 u - \Delta z + u &=  \mathcal{D}
    \\
    \frac{\alpha}{2} f + z &= 0
    \\
    \Delta u + f &= 0.
  \end{split}
\end{equation}

Unlike other approaches, we do not solve the KKT conditions
explicitly, instead we examine the Lagranigan formulation directly. We
seek the saddle points of the Lagrangian $L$, that satisfy $(u^*,f^*,
z^*) \in \Hspace(\W) \times \leb2(\W) \times \leb2_0(\W)$ such that
\begin{equation}\label{eq:u*,f*,z* defn}
  (u^*,f^*, z^*)
  =
  \argmin\limits_{u,f \in  \Hspace(\W)\times  \leb2(\W)}\argmax\limits_{z \in \leb2_0}\ L(u,f,z).
\end{equation}

\subsection{Iterative Approximation}\label{iter_PDE}

To motivate our neural network methods introduced in the next section,
we first define the iterates at the PDE level.  We propose an
iterative method for the solution of (\ref{eq:minmax}) through an
Uzawa algorithm at the continuum. The Uzawa algorithm iteratively
generates sequences of functions $\setbra{u^k}_{k=0}^\infty
\subset \Hspace(\W)$, $\setbra{f^k}_{k=0}^\infty \subset \leb{2}(\W)$, and
$\setbra{z^k}_{k=0}^\infty \subset \leb2_0(\W)$. These are defined
through an initial guess $z^0 \in \leb2_0(\W)$. Then for a step size
$\rho > 0$ that will be specified, for $k = 0, 1, 2, \ldots$, find
$u^k, f^k$ such that
   \begin{equation}\label{eq:Uzawa scheme Poisson}
     \begin{split}
       (u^k, f^k) &= \argmin_{u, f \in \Hspace(\W) \times \leb{2}(\W)} L(u, f, z^k)
       \\
       z^{k+1} &= z^k + \rho K(u^k, f^k).
     \end{split}
   \end{equation}

   At each iteration, we first solve a minimisation problem that can
   be quantified through the first-order optimality conditions:
   \begin{equation}
     \label{eq:uzawaEL}
     \begin{split}
       \frac{\alpha}{2} \Delta^2 u + u &= \mathcal{D} - \Delta z^k
       \\
       f &= - \frac 2 \alpha z^k.
  \end{split}
\end{equation}
We emphasise however that our intention is to discretise \eqref{eq:Uzawa scheme Poisson} directly, and not via the optimality conditions 
\eqref{eq:uzawaEL}.

\subsection{Other Modified Lagrangian Approaches}
\label{sec:modifiedLagragian}

An issue in posing classical iterative approaches to the solution of
(\ref{eq:minmax}) is the relative strength of the functional
$\widehat{J}$ with respect to the regularity requirements of the PDE
constraint. Classical methods for solving PDE-constrained optimisation
problems include penalty methods, adjoint-based methods, and augmented
Lagrangian methods. Penalty methods transform the constrained problem
into an unconstrained one by adding a penalty term to the cost
functional that enforces the PDE constraint. For example
\begin{equation}
  \frac{1}{2} \Norm{u-\mathcal{D}}^2_{\leb{2}(\W)}
  +
  \frac{\alpha}{2}\Norm{f}^2_{\leb{2}(\W)}
  +
  \frac \beta 2 \Norm{K(u,f)}^2_{\leb{2}(\W)}
  \to \text{ min }.
\end{equation}
While this approach is straightforward, it often leads to
ill-conditioned optimisation problems, particularly when the penalty
parameter is large, as required to accurately enforce the constraint.

Adjoint-based methods, on the other hand, involve solving the adjoint
state equation to compute the gradient of the cost functional with
respect to the control variables. This gradient is then used in
gradient-based optimisation algorithms such as gradient descent or
quasi-Newton methods. In the simplest case, such an iterative method
might resemble a Gauss-Siedel approach. For fixed $f^{0}$ find
$u^{k+1}$ such that
\begin{equation}
  \begin{split}
    -\Delta u^{k+1} &= f^{k}
    \\
    - \Delta z^{k+1} &=  \mathcal{D} - u^{k+1}
    \\
    f^{k+1} &= -\frac{1}{\alpha} z^{k+1}.
  \end{split}  
\end{equation}
Although effective, adjoint-based methods can be computationally
expensive and slow to converge, especially for complex PDEs or
large-scale problems.

Augmented Lagrangian methods combine the advantages of penalty methods
and Lagrange multipliers by adding both penalty terms and Lagrange
multipliers to the cost functional. This approach improves the
conditioning of the optimisation problem compared to penalty methods
alone but still requires careful tuning of parameters and can be
computationally intensive.

The augmented Lagrangian for this problem is
\begin{equation}
  \label{eq:AugmentedLagrangian}
  \mathcal{L}_\beta(u,f,z)
  :=
  \widehat{J}(u,f)
  +
  \ltwop{K(u,f)}{z}
  +
  \frac{\beta}{2} \Norm{K(u,f)}^2_{\leb{2}(\W)}.
\end{equation}

\begin{remark}[Comparison to the augmented Lagrangian method]
  While we do not use the formulation directly, it is informative to
  compare the KKT conditions (\ref{eq:kkt-uzawa}) with those of the
  augmented Lagrangian method (\ref{eq:AugmentedLagrangian}). They are
  given as
  \begin{equation}
    \label{eq:KKTAugmentedLagrangian}
    \begin{split}
      \beta \Delta (\Delta u + f) + \Delta z + u &=  \mathcal{D}
      \\
      \beta (\Delta u + f) + \alpha f + z &= 0
      \\
      \Delta u + f &= 0.
    \end{split}
  \end{equation}
  It is important to note that each KKT system
  (\ref{eq:KKTAugmentedLagrangian}), (\ref{eq:kkt-uzawa}) and
  (\ref{eq:frechet_derivatives}) are consistent with one another and,
  for this model problem assuming smoothness and compatible boundary
  conditions, the control and dual variable can be eliminated to show
  that the state equation satisfies
  \begin{equation}
    \alpha \Delta^2 u + u = \mathcal D.
  \end{equation}
\end{remark}

\begin{remark}[Augmented Lagrangian and an Uzawa iteration]
  The Euler-Lagrange equation arising from the Uzawa minimisation
  problem (\ref{eq:Uzawa scheme Poisson}) can be easily written for
  the augmented Lagrangian method.

  Indeed, for an initial guess $z^{0}$, for $k = 0, 1, 2, \ldots$ find
  $u^{k+1}, f^{k+1}$ such that
  \begin{equation}
    \begin{split}
      (u^{k+1}, f^{k+1})
      =
      \argmin_{u, f} \widehat{J}(u, f)
      +
      \ltwop{K(u,f)}{z^{k}}
      +
      \frac{\beta}{2} \Norm{K(u,f)}^2_{\leb{2}(\W)}.
    \end{split}
  \end{equation}
  and update
  \begin{equation}
    z^{k+1} = z^{k} + \beta K(u^{k+1}, f^{k+1}).
  \end{equation}
  
  Note that the state, control update is equivalent to solving the
  Euler-Lagrange equations 
  \begin{equation}
    \label{eq:augEL}
    \begin{split}
      \beta \Delta(\Delta u + f) + u  &= \mathcal{D} - \Delta z^{k}
      \\
      \beta (\Delta u + f) + \alpha f &= - z^{k}.
    \end{split}
  \end{equation}

  The augmented Lagrangian method combines the penalty term with the
  Lagrange multipliers, which requires careful tuning of the penalty
  parameters to balance constraint enforcement and problem
  conditioning.

  Both iterative schemes (\ref{eq:uzawaEL}) and (\ref{eq:augEL}) fit
  into the analytic framework we present, thereby providing a
  theoretical understanding on the convergence behaviour of this
  method observed in \cite{PINNs_augmented_Lagrangian}.  However, our
  focus in this work is on \eqref{eq:Uzawa scheme Poisson}.
\end{remark}

\section{Convergence Analysis}
\label{sec:conv}

As a first step towards a complete convergence analysis, we
demonstrate the convergence of our iterative scheme at the PDE level,
assuming exact solutions for the PDEs and energy minimisation
problems. We then generalise this convergence to restricted function
spaces. In future work, we will extend our analysis to account for
errors due to neural network approximations. Here, we present the
convergence result, with the proof deferred to \S \ref{sec:proof}.
The proofs are based on appropriate adaptations of classical arguments
regarding the convergence of Uzawa algorithms,
see \cite{CiarletMiaraThomas:1989}.

\begin{theorem}
  \label{the:convergence}
  Assume that $\W \subset \mathbb{R}^d$ is a convex polygonal domain.
  Let the sequences $\setbra{u^k}_{k=0}^\infty \subset \Hspace(\W)$ and
  $\setbra{f^k}_{k=0}^\infty \subset \leb2(\W)$ be generated through
  the iterative method \eqref{eq:Uzawa scheme Poisson}. Suppose $u^*
  \in \Hspace(\W)$ and $f^* \in \leb2(\W)$ are critical points of
  (\ref{eq:u*,f*,z* defn}). Then, for $\rho \in (0, \tfrac \alpha 2)$,
  we have that
  \begin{equation}
    u^k \rightarrow u^* \text{ in } \sobh2(\W),
    \qquad 
    f^k \rightarrow f^* \text{ in } \leb2(\W).
  \end{equation}
\end{theorem}

\begin{remark}[Generalised elliptic operators]
  The results presented hold for more general elliptic operators,
  indeed any second order elliptic problem admitting
  $u\in\sobh{2}(\W)$ regularity results falls under the same
  framework. It can also be extended to those problems that only admit
  a weak solution, on a non-convex domain say, albeit one only expects
  convergence in $\sobh1(\W)$. It is interesting to note that the
  admissible values of $\rho$ have no dependence on the regularity
  constant.
\end{remark}

\subsection{Non-negativity constraint}\label{sec:non-negative}
We define a restricted function space
\begin{equation}\label{eq:defn_constrained_functionspace}
  \Hspace(\W, \reals^+)
  :=
  \setbra{\phi \in \Hspace(\W) : \phi \geq 0 }, \qquad
  \leb{2}_0(\W, \reals^+)
  :=
  \setbra{\psi \in \leb{2}_0 : \psi \geq 0 }
\end{equation}
and consider the PDE constrained optimisation problem
\begin{equation}\label{eq:altcostfunctional_constrained}
  \begin{split}
    (u^*, f^*)
    &=
    \argmin_{u, f \in \Hspace(\W)^+  \times \leb2(\W,\real^+)} {J}(u,f)
  \end{split}
\end{equation}
subject to the constraint given in \eqref{eq:altcostfunctional3}. This
may be formulated as a constrained saddle-point problem,
\begin{equation}\label{eq:box_constraint_cont}
    \begin{split}
        u^*,f^*,z^* &= \argmin\limits_{u, f \in \Hspace(\W, \reals^+)  \times \leb2(\W,\real^+)} \argmax\limits_{z \in \leb{2}_0(\W, \reals^+) } L(u,f,z).
    \end{split}
\end{equation}
Similar to the unconstrained case, we may construct an iterative Uzawa
scheme which is derived from the first order optimality
conditions. Let $P^{+}:\leb2(\W) \rightarrow \leb2(\W, \real^+)$ be
the canonical projection operator given by
\begin{equation}\label{eq:proj_map}
  P^+\phi(\vec x) := \max\setbra{\phi(\vec x), 0} \qquad \forall \vec x \in \W.
\end{equation}
For an initial guess $z^0\in \leb{2}_0(\W, \reals^+)$ and step size
$\rho>0$, for $k = 0, 1, 2, \cdots$ we seek $u^k, f^k$ such that
\begin{equation}\label{eq:box_constraint_iter}
  \begin{split}
    u^k,f^k &= \argmin\limits_{u,f \in \Hspace(\W, \reals^+) \times \leb2(\W,\real^+) } L(u,f,z^k) \\
    z^{k+1} &= P^{+}\bra{z^k + \rho K(u^k,f^k)}.
  \end{split}
\end{equation}

\begin{theorem}\label{the:box_constraint_converge}
  Let $\W \subset \real^d$ be a convex polygonal domain. Let the
  sequences $\setbra{u^k}_{k=0}^\infty \subset \Hspace(\W, \reals^+)$ and
  $\setbra{f^k}_{k=0}^\infty \subset \leb2(\W,\real^+)$ be generated
  through the iterative method \eqref{eq:box_constraint_iter}. Suppose
  $u^* \in \Hspace(\W, \reals^+)$ and $f^* \in \leb2(\W,\real^+)$ are critical
  points of \eqref{eq:box_constraint_cont}. Then, for $\rho \in (0,
  \tfrac \alpha 2)$, we have that
  \begin{equation}
    u^k \rightarrow u^* \text{ in } \sobh2(\W),
    \qquad 
    f^k \rightarrow f^* \text{ in } \leb2(\W).
  \end{equation}
\end{theorem}

\section{Deep Uzawa Algorithm}
\label{sec:deepuzawa}

This section outlines the methodology for constructing a neural
network approximation scheme, demonstrated with a series of examples.

\subsection{Neural Network Function Approximation}

We consider approximations to the PDE constrained optimisation problem
generated by neural networks. The structure described is for
illustrative purposes, and our results are not dependent on specific
neural network architectures. A deep neural network maps each point $x
\in \W$ to a value $w_\theta(x) \in \mathbb{R}$ through the process
\begin{equation}
  \label{eq:NN1}
  w_\theta(x) = \mathcal{C}_L(x)
  :=
  {C}_L \circ \sigma \circ {C}_{L-1} \circ \cdots \circ \sigma \circ {C}_1(x) \quad \forall x \in \W.
\end{equation}
The mapping $\mathcal{C}_L: \mathbb{R}^d \rightarrow \mathbb{R}^2$ in
our application. This $\mathcal{C}_L$ is a neural network with $L$
layers and an activation function $\sigma$. To define $w_\theta(x)$
for all $x \in \W_D$, we use the same $\mathcal{C}_L$, thus
$w_\theta(\cdot) = \mathcal{C}_L(\cdot)$.

Each map $\mathcal{C}_L$ is characterised by intermediate (hidden)
layers $C_k$, which are affine transformations of the form
\begin{equation}
  \label{eq:NN2}
  C_k y = W_k y + b_k, \quad \text{where } W_k \in \mathbb{R}^{d_{k+1} \times d_k}, \, b_k \in \mathbb{R}^{d_{k+1}}.
\end{equation}
Here, the dimensions $d_k$ can vary with each layer $k$, and
$\sigma(y)$ represents the vector with the same number of components
as $y$, where $\sigma(y)_i = \sigma(y_i)$. The index $\theta$
collectively represents all parameters of the network $\mathcal{C}_L$,
namely $W_k, b_k$, for $k = 1, \ldots, L$. The set of all networks
$\mathcal{C}_L$ with a given structure (fixed $L, d_k, k = 1, \ldots,
L$) of the form (\ref{eq:NN1})--(\ref{eq:NN2}) is called
$\mathcal{N}$.

The total dimension (number of degrees of freedom) of $\mathcal{N}$ is
\begin{equation}
  \dim \mathcal{N} = \sum_{k=1}^L d_{k+1} (d_k + 1).
\end{equation}
We define the space of functions as
\begin{equation}
  \mathcal{V}_N
  =
  \{w_\theta : \W \rightarrow \mathbb{R}, \text{ where }
  w_\theta(x) = \mathcal{C}_L(x) \text{ for some } \mathcal{C}_L \in \mathcal{N}\}.
\end{equation}
It is important to note that $\mathcal{V}_N$ is not a linear space. We denote by
\begin{equation}
  \Theta = \{\theta : w_\theta \in \mathcal{V}_N\}
\end{equation}
and note that $\Theta$ is a linear subspace of $\mathbb{R}^{\dim
  \mathcal{N}}$.

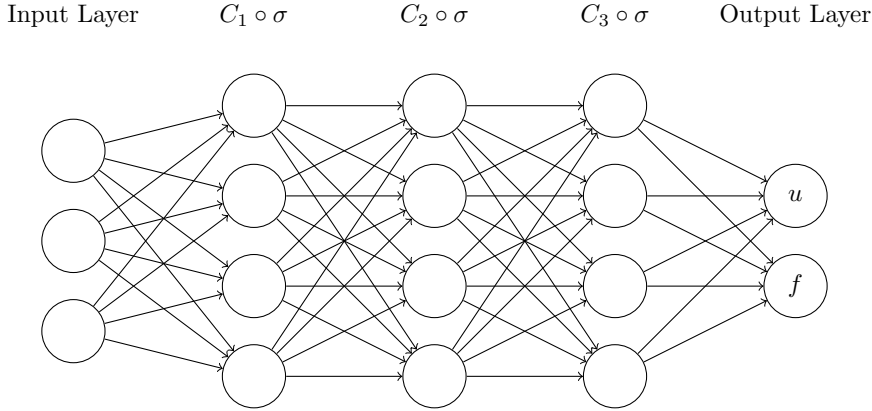
\begin{figure}[h!]  
  \begin{tikzpicture}[scale=1.2, transform shape]

% Input layer
\foreach \i in {1, 2, 3} {
    \node[circle, draw=black, fill=white, minimum size=0.7cm] (I\i) at (0, -\i) {};
    %\node[scale=0.8] at (-0.5, -\i) {Input \i};
}

% Hidden layer 1
\foreach \i in {1, 2, 3, 4} {
    \node[circle, draw=black, fill=white, minimum size=0.7cm] (H1\i) at (2, -\i+0.5) {};
}

% Hidden layer 2
\foreach \i in {1, 2, 3, 4} {
    \node[circle, draw=black, fill=white, minimum size=0.7cm] (H2\i) at (4, -\i+0.5) {};
}

% Hidden layer 3
\foreach \i in {1, 2, 3, 4} {
    \node[circle, draw=black, fill=white, minimum size=0.7cm] (H3\i) at (6, -\i+0.5) {};
}

% Output layer
\foreach \i in {1, 2} {
    \node[circle, draw=black, fill=white, minimum size=0.7cm] (O\i) at (8, -\i-0.5) {};
}
\node[scale=0.8] at (8., -1-0.5) {$u$};
\node[scale=0.8] at (8., -2-0.5) {$f$};

% Input to Hidden 1
\foreach \i in {1, 2, 3} {
    \foreach \j in {1, 2, 3, 4} {
        \draw[->] (I\i) -- (H1\j);
    }
}

% Hidden 1 to Hidden 2
\foreach \i in {1, 2, 3, 4} {
    \foreach \j in {1, 2, 3, 4} {
        \draw[->] (H1\i) -- (H2\j);
    }
}

% Hidden 2 to Hidden 3
\foreach \i in {1, 2, 3, 4} {
    \foreach \j in {1, 2, 3, 4} {
        \draw[->] (H2\i) -- (H3\j);
    }
}

% Hidden 3 to Output
\foreach \i in {1, 2, 3, 4} {
    \foreach \j in {1, 2} {
        \draw[->] (H3\i) -- (O\j);
    }
}

% Labels
\node[scale=0.8] at (0, 0.5) {Input Layer};
\node[scale=0.8] at (2, 0.5) {$C_1 \circ \sigma$};
\node[scale=0.8] at (4, 0.5) {$C_2 \circ \sigma$};
\node[scale=0.8] at (6, 0.5) {$C_3 \circ \sigma$};
\node[scale=0.8] at (8, 0.5) {Output Layer};

\end{tikzpicture}
  \caption{Illustration of the architecture within $\mathcal C_L$.}
\end{figure}
The Deep Uzawa algorithm minimises residual-type functionals over the
discrete space $\mathcal{V}_N$.
\begin{definition}[Deep-$\mathcal{V}_N$ minimiser]
  For fixed $z^k$, assume that the problem
  \begin{equation}
    \min_{(u,f) \in \mathcal{V}_N} L(u,f,z^k)
  \end{equation}
  has a solution $(u^*,f^*) \in \mathcal{V}_N$. We call $(u^*,f^*)$ a
  deep-$\mathcal{V}_N$ minimiser of $L(\cdot , \cdot, z^k)$.
\end{definition}

One key difficulty in this problem is that $\mathcal{V}_N$ is not a
linear space. Computationally, it can be equivalently formulated as a
minimisation problem in $\mathbb{R}^{\dim \mathcal{N}}$ by considering
$\theta$ as the parameter vector to be identified through
\begin{equation}
  \min_{\theta \in \Theta} L(u_\theta, f_\theta, z^k).
\end{equation}
Although this is well defined as a discrete minimisation problem, it
is generally non-convex with respect to $\theta$ even if the
functional $L(u,f,z)$ is convex with respect to $(u,f)$ such as in the
penalty methods. This is one of the primary technical challenges in
machine learning algorithms.

\subsection{Training within the Deep Uzawa framework}

To implement the Deep Uzawa iteration, we need computable discrete
versions of the Lagrangian $L(u_\theta, f_\theta, z^k)$. This can be
achieved in different ways. A common approach is to use appropriate
quadrature for integrals over $\W$ (training through quadrature). To
fix ideas, such a quadrature requires a set $\mathcal{K}_h$ of
discrete points $y \in \mathcal{K}_h$ and corresponding nonnegative
weights $w_y$ such that
\begin{equation}
  \sum_{y \in \mathcal{K}_h} w_y g(y) \approx \int_{\W_D} g(x) \, dx.
\end{equation}
Then one can define the discrete functional
\begin{equation}\label{eq:discretelagrangian}
  L_{Q}(u,f,z)
  =
  \sum_{y \in \mathcal{K}_h}
  w_y
  \qp{
    \frac{1}{2} \norm{u(y)-\mathcal{D}(y)}^2
    +
    \frac{\alpha}{4}\norm{f(y)}^2
    +
    \frac{\alpha}{4}
    \norm{\Delta u(y)}^2
    +
    z(y) K(u(y), f(y))
  }.
\end{equation}
Notice that both deterministic and probabilistic quadrature rules are
possible, yielding different final algorithms.

In this work, we shall not consider in detail the influence of the
quadrature (and hence of the training) on the stability and
convergence of the algorithms. This requires a much more involved
technical analysis and will be the subject of future
research. 

\subsection{The Deep Uzawa Iteration}

The Deep Uzawa algorithm combines the classical Uzawa iteration with
neural network-based minimisation. This approach involves alternating
between an inner loop consisting of several iterations of minimisation
and an outer loop which corresponds to a single Uzawa update step as
described in Algorithm \ref{alg:uzawa}. Note the total number of
Epochs is the product $N_{\text{SGD}}N_{\text{Uz}}$.

\begin{algorithm}[h!]
  \caption{Deep Uzawa Iteration
    \label{alg:uzawa}
  }
  \begin{algorithmic}[1]
    \Require Initial guess $z^0 \in \leb{2}_0(\W)$, Uzawa step size $\rho > 0$, number of Uzawa steps $N_{\text{Uz}}$, number of Stochastic Gradient Descent iterations $N_{\text{SGD}}$, learning rate $\eta$.
    \State $k \gets 0$
    \State Initialise neural network parameters $\theta^0$
    \For{$k = 0$ to $N_{\text{Uz}} - 1$}
        \For{$m = 0$ to $N_{\text{SGD}}-1$}
        \State Compute stochastic gradient $\nabla_\theta L_Q(u_\theta^m, f_\theta^m, z^k)$
        \State Update parameters: $\theta^{m+1} \gets \theta^m - \eta \nabla_\theta L_Q(u_\theta^m, f_\theta^m, z^k)$
        \EndFor
        \State $(u^k, f^k) \gets (u_\theta^{N_{\text{SGD}}}, f_\theta^{N_{\text{SGD}}})$
      \State Update Lagrange multiplier: $z^{k+1} \gets z^k + \rho K(u^k, f^k)$
      \State $k \gets k + 1$
    \EndFor
  \end{algorithmic}
\end{algorithm}

\section{Numerical Experiments: Linear Constraints}\label{sec:PDECon_lin_constr}

We will now detail some numerical experiments. In each case we use the
PyTorch Adam class optimiser \cite{stochasticADAMPyTorch} with
learning rate $\eta = 10^{-3}$. Additionally, unless otherwise
specified, we let $N_{\text{SGD}} = 40$ and $N_{\text{Uz}} =
500$. Finally, we consider $N=201$ uniformly distributed collocation
points that form a Cartesian grid, unless otherwise specified. 

\subsection{Example: Smooth Target}\label{sec: 1D Sine function target}
To begin, we consider the one-dimensional interval $\W = (0,1)$ and target function
\begin{equation}
  \mathcal{D}(x) = \bra{1+\alpha \pi^4}\sin{\pi x}.
  \label{eq: 1D sine function}
\end{equation}  
Here we compute
\begin{equation}\label{eq: Exact solution 1D sine}
  u^*(x,y) = \sin{\pi x}, \qquad f^*(x,y) = \pi^2 \sin{\pi x}.
\end{equation}
In Figure \ref{fig:sine_realisation_plot} we show the neural network
approximation to both state and control, and the corresponding
pointwise errors, for a particular choice of $\alpha$. 
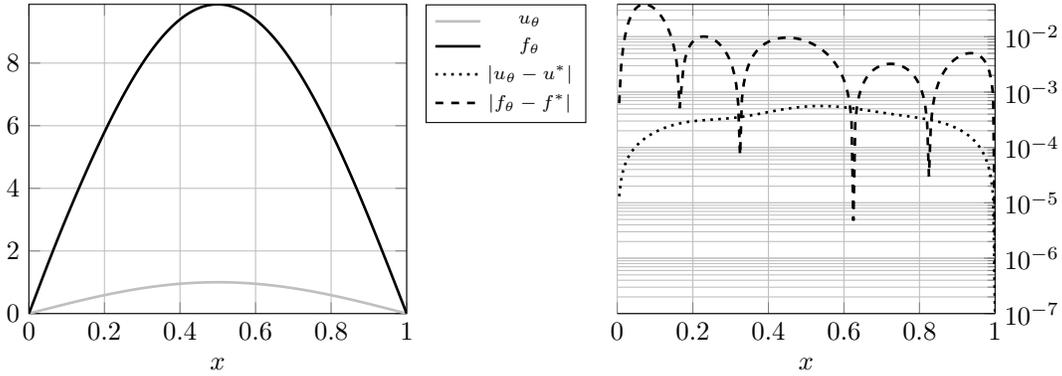
\begin{figure}[h!]
  %\ifthenelse{\boolean{showTikz}}
  {
    \begin{tikzpicture}
    \begin{groupplot}[
        group style={
            group size=2 by 1,
            horizontal sep=2.8cm, % Adjusted horizontal separation
            vertical sep=2cm,
        },
        enlargelimits=false,
        axis line style={draw=black} % Box around both plots
    ]
 
    % Left plot
      \nextgroupplot[
        width=0.4\textwidth, % Adjusted figure width
        xlabel={$x$},
        grid=both,
        legend style={at={(1.05,1)}, anchor=north west, font=\scriptsize}, % Smaller legend
        axis line style={draw=black} % Box around the left plot
    ]
    % Manually specified plots
    \addlegendentry{$u_{\theta}$}
    \addplot[mark=none, color={black!25}, line width=1.0] table [col sep=comma, header=true, x expr=\coordindex/200, y index=0] {Plots/1D_Sine/Vary_alpha/Vary_alpha_4_nPts_201/State.csv};
    %%%%%%%%%%%
    \addlegendentry{$f_{\theta}$}
    \addplot[mark=none, color={black!100}, line width=1.0] table [col sep=comma, header=true, x expr=\coordindex/200, y index=0] {Plots/1D_Sine/Vary_alpha/Vary_alpha_4_nPts_201/Control.csv};
    %%%%%%%%%%%
    \addlegendentry{$\abs{u_{\theta}-u^*}$}
    \addplot[mark=none, dotted, color={black!100}, line width=1.0] coordinates {
    (0, 0)
    (0, 0)}; 
    %%%%%%%%%%%
    \addlegendentry{$\abs{f_{\theta}-f^*}$}
    \addplot[mark=none, dashed, color={black!100}, line width=1.0] coordinates {
    (0, 0)
    (0, 0)}; 
    
    % Right plot
      \nextgroupplot[
        width=0.4\textwidth, % Adjusted figure width
        xlabel={$x$},
        axis y line*=right,
        ymode=log,
        log basis y={10},
        ymin = 1e-7,
        xmin = 0,
        xmax = 1,
        grid=both,
        legend style={at={(1.05,1)}, anchor=north west, font=\scriptsize}, % Smaller legend
        axis line style={draw=black} % Box around the left plot
    ]
    \addplot[mark=none, dotted, color={black!100}, line width=1.0] table [col sep=comma, header=true, x expr=\coordindex/200, y index=0, dotted] {Plots/1D_Sine/Vary_alpha/Vary_alpha_4_nPts_201/State_pointwise_error.csv};
    \addplot[mark=none, dashed, color={black!100}, line width=1.0] table [col sep=comma, header=true, x expr=\coordindex/200, y index=0] {Plots/1D_Sine/Vary_alpha/Vary_alpha_4_nPts_201/Control_pointwise_error.csv};
    
    \end{groupplot}
\end{tikzpicture}
  }{
  }
  \caption{\label{fig:sine_realisation_plot} Example \ref{sec: 1D Sine
      function target}. Plots of the state and control (left), and the
    pointwise error (right), where the exact solutions $u^*, f^* \in
    \mathcal{V}_N$ are given in equation \eqref{eq: Exact solution 1D
      sine}, for $\alpha = 10^{-4}$.
      }
\end{figure}

In Figure \ref{fig:1D_sine_error_alpha} we show the $\leb2$-error of
the neural network approximation, for both state and control, against
update number for varying $\alpha$ values. The rate of convergence of the
state, with respect to update number, is is robust with respect to changes in 
$\alpha$. However, the rate of convergence of the control decreases, with
respect to update number, as $\alpha$ decreases. This is because the
component of the loss functional that corresponds to the control is
scaled with respect to $\alpha$. 
\begin{figure}[h!]
  %\ifthenelse{\boolean{showTikz}}
  {
    \begin{tikzpicture}

    \begin{groupplot}[
        group style={
            group size=2 by 1,
            horizontal sep=2.8cm, 
            vertical sep=2cm,
        },
        enlargelimits=false,
        axis line style={draw=black} 
    ]

    % Left plot
      \nextgroupplot[
        width=0.39\textwidth,
        xlabel={Update Number $N_{Uz}$},
        ylabel={$\| u_\theta -u^*\|_{L^2(\Omega)}$},
        grid=both,
        ymode=log,
        log basis y={10},
        legend style={at={(1.05,1)}, anchor=north west, font=\scriptsize},
        axis line style={draw=black}
    ]
      \addplot[mark=none, color={black!10}, line width=1.5] table [col sep=comma, header=true, x expr=\coordindex, y index=0] {Plots/1D_Sine/Vary_alpha/Vary_alpha_0_nPts_201/Error.csv};
      \addlegendentry{$\alpha=10^{0}$}
      \addplot[mark=none, color={black!30}, line width=1.5] table [col sep=comma, header=true, x expr=\coordindex, y index=0] {Plots/1D_Sine/Vary_alpha/Vary_alpha_2_nPts_201/Error.csv};
      \addlegendentry{$\alpha=10^{-2}$}
      \addplot[mark=none, color={black!50}, line width=1.5] table [col sep=comma, header=true, x expr=\coordindex, y index=0] {Plots/1D_Sine/Vary_alpha/Vary_alpha_4_nPts_201/Error.csv};
      \addlegendentry{$\alpha=10^{-4}$}
      \addplot[mark=none, color={black!70}, line width=1.5] table [col sep=comma, header=true, x expr=\coordindex, y index=0] {Plots/1D_Sine/Vary_alpha/Vary_alpha_6_nPts_201/Error.csv};
      \addlegendentry{$\alpha=10^{-6}$}    
      \addplot[mark=none, color={black!90}, line width=1.5] table [col sep=comma, header=true, x expr=\coordindex, y index=0] {Plots/1D_Sine/Vary_alpha/Vary_alpha_8_nPts_201/Error.csv};
      \addlegendentry{$\alpha=10^{-8}$}
      \addplot[mark=none, color={black!110}, line width=1.5] table [col sep=comma, header=true, x expr=\coordindex, y index=0] {Plots/1D_Sine/Vary_alpha/Vary_alpha_10_nPts_201/Error.csv};
      \addlegendentry{$\alpha=10^{-10}$}
    % Left plot
    \nextgroupplot[
      width=0.39\textwidth,
      xlabel={Update number $N_{Uz}$},
      ylabel={$\| f_\theta -f^*\|_{L^2(\Omega)}$},
      grid=both,
      ymode=log,
      axis y line*=right,
      log basis y={10},
      legend style={at={(1.05,1)}, anchor=north west, font=\scriptsize},
      axis line style={draw=black} 
    ]
    \addplot[mark=none, color={black!10}, line width=1.5] table [col sep=comma, header=true, x expr=\coordindex, y index=1] {Plots/1D_Sine/Vary_alpha/Vary_alpha_0_nPts_201/Error.csv};
    \addplot[mark=none, color={black!30}, line width=1.5] table [col sep=comma, header=true, x expr=\coordindex, y index=1] {Plots/1D_Sine/Vary_alpha/Vary_alpha_2_nPts_201/Error.csv};
    \addplot[mark=none, color={black!50}, line width=1.5] table [col sep=comma, header=true, x expr=\coordindex, y index=1] {Plots/1D_Sine/Vary_alpha/Vary_alpha_4_nPts_201/Error.csv};
    \addplot[mark=none, color={black!70}, line width=1.5] table [col sep=comma, header=true, x expr=\coordindex, y index=1] {Plots/1D_Sine/Vary_alpha/Vary_alpha_6_nPts_201/Error.csv};
    \addplot[mark=none, color={black!90}, line width=1.5] table [col sep=comma, header=true, x expr=\coordindex, y index=1] {Plots/1D_Sine/Vary_alpha/Vary_alpha_8_nPts_201/Error.csv};
    \addplot[mark=none, color={black!110}, line width=1.5] table [col sep=comma, header=true, x expr=\coordindex, y index=1] {Plots/1D_Sine/Vary_alpha/Vary_alpha_10_nPts_201/Error.csv};  
    \end{groupplot}

\end{tikzpicture}
  }{}
  \caption{Example \ref{sec: 1D Sine function target}. Plots of the
    $\leb2$-error of approximation of state and control as a function
    of $N_{Uz}$ for various regularisations $\alpha \in [10^{-10}, 1]$.
  }
  \label{fig:1D_sine_error_alpha}
\end{figure}
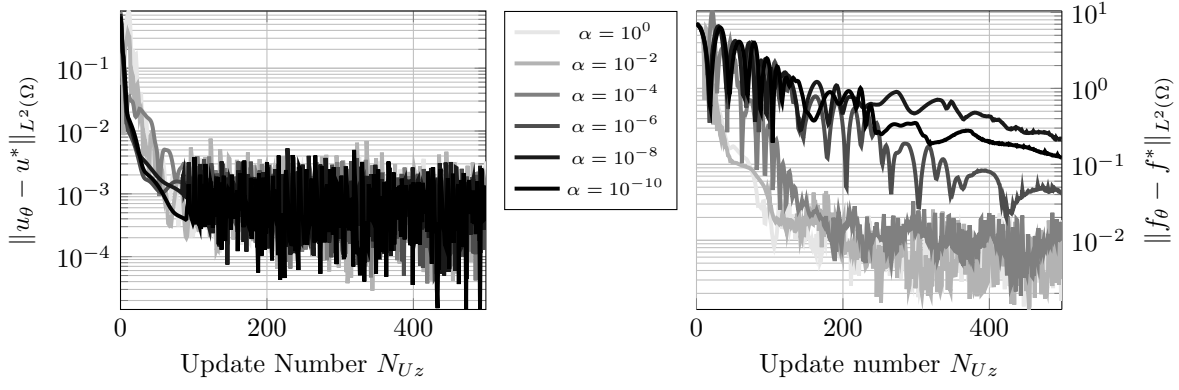

In Figure \ref{fig:1D_sine_error_NSGD} we show the $\leb2$-error of
the neural network approximation, for both state and control, against
$N_{Uz}$ for varying $N_{SGD}$ values. This is quantifying the number
of inner stochastic gradient descent iterations per Uzawa update.
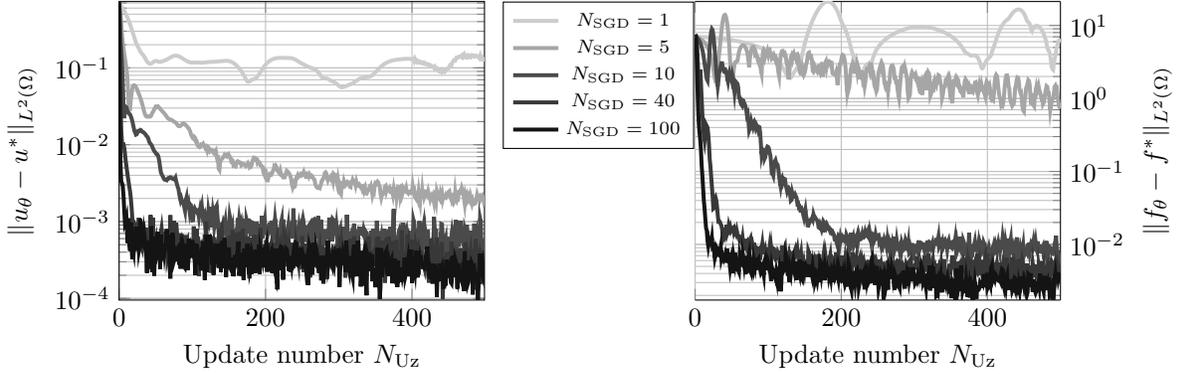
\begin{figure}[h!]
  %\ifthenelse{\boolean{showTikz}}
  {
    \begin{tikzpicture}
  
    \begin{groupplot}[
        group style={
            group size=2 by 1,
            horizontal sep=2.8cm, 
            vertical sep=2cm,
        },
        enlargelimits=false,
        axis line style={draw=black}
    ]

    \nextgroupplot[
        width=0.39\textwidth,
        xlabel={Update number $N_{\text{Uz}}$},
        ylabel={$\| u_\theta - u^* \|_{L^2(\Omega)}$},
        grid=both,
        ymode=log,
        log basis y={10},
        legend style={at={(1.05,1)}, anchor=north west, font=\scriptsize},
        axis line style={draw=black}
    ]
    \addplot[mark=none, color={black!20}, line width=1.5] table [col sep=comma, header=true, x expr=\coordindex, y index=0] {Plots/1D_Sine/Vary_NSGD/Vary_alpha_4_nPts_201_EpUp_1/Error.csv};
    \addlegendentry{$N_{\text{SGD}}=1$}
    \addplot[mark=none, color={black!35}, line width=1.5] table [col sep=comma, header=true, x expr=\coordindex, y index=0] {Plots/1D_Sine/Vary_NSGD/Vary_alpha_4_nPts_201_EpUp_5/Error.csv};
    \addlegendentry{$N_{\text{SGD}}=5$}
    \addplot[mark=none, color={black!71}, line width=1.5] table [col sep=comma, header=true, x expr=\coordindex, y index=0] {Plots/1D_Sine/Vary_NSGD/Vary_alpha_4_nPts_201_EpUp_10/Error.csv};
    \addlegendentry{$N_{\text{SGD}}=10$}
    \addplot[mark=none, color={black!78}, line width=1.5] table [col sep=comma, header=true, x expr=\coordindex, y index=0] {Plots/1D_Sine/Vary_NSGD/Vary_alpha_4_nPts_201_EpUp_40/Error.csv};
    \addlegendentry{$N_{\text{SGD}}=40$}
    \addplot[mark=none, color={black!92}, line width=1.5] table [col sep=comma, header=true, x expr=\coordindex, y index=0] {Plots/1D_Sine/Vary_NSGD/Vary_alpha_4_nPts_201_EpUp_100/Error.csv};
    \addlegendentry{$N_{\text{SGD}}=100$}

    \nextgroupplot[
      width=0.39\textwidth,
        xlabel={Update number $N_{\text{Uz}}$},
        ylabel={$\| f_\theta - f^* \|_{L^2(\Omega)}$},
        grid=both,
        ymode=log,
        xmin = 0,
        xmax = 500,
        axis y line*=right,
        log basis y={10},
        legend style={at={(1.05,1)}, anchor=north west, font=\scriptsize},
        axis line style={draw=black}
    ]
    \addplot[mark=none, color={black!20}, line width=1.5] table [col sep=comma, header=true, x expr=\coordindex, y index=1] {Plots/1D_Sine/Vary_NSGD/Vary_alpha_4_nPts_201_EpUp_1/Error.csv};
    \addplot[mark=none, color={black!35}, line width=1.5] table [col sep=comma, header=true, x expr=\coordindex, y index=1] {Plots/1D_Sine/Vary_NSGD/Vary_alpha_4_nPts_201_EpUp_5/Error.csv};
    \addplot[mark=none, color={black!71}, line width=1.5] table [col sep=comma, header=true, x expr=\coordindex, y index=1] {Plots/1D_Sine/Vary_NSGD/Vary_alpha_4_nPts_201_EpUp_10/Error.csv};
    \addplot[mark=none, color={black!78}, line width=1.5] table [col sep=comma, header=true, x expr=\coordindex, y index=1] {Plots/1D_Sine/Vary_NSGD/Vary_alpha_4_nPts_201_EpUp_40/Error.csv};
    \addplot[mark=none, color={black!92}, line width=1.5] table [col sep=comma, header=true, x expr=\coordindex, y index=1] {Plots/1D_Sine/Vary_NSGD/Vary_alpha_4_nPts_201_EpUp_100/Error.csv};

    \end{groupplot}

\end{tikzpicture}
  }{
  }
  \caption{Example \ref{sec: 1D Sine function target}. Plots of the
    $\leb2$-error of approximation of state and control for $\alpha =
    10^{-4}$ and $N_{\text{SGD}} \in [1, 100]$.  }
    \label{fig:1D_sine_error_NSGD}
\end{figure}

We compare our scheme to an augmented Lagrangian approach
\cite{PINNs_augmented_Lagrangian}, equation
\eqref{eq:AugmentedLagrangian}. Recall the discrete functional $L_Q$,
equation \eqref{eq:discretelagrangian}; for a given $\beta >0$ we
define
\begin{equation}
  \mathcal{L}_{\beta, Q}(u,f,z)
  =
  L_{Q}(u,f,z) + \dfrac{\beta}{2}\sum\limits_{y \in \mathcal{K}_h}w_y (K(u(y),f(y))^2.
\end{equation}
Consider the following minimisation scheme:
\begin{equation}\label{eq: Augemented Lagrangian Poisson}
  \begin{split}
    (u^k_{\beta, \theta}, f^k_{\beta, \theta}) &= \argmin_{u, f \in \mathcal{V}_N \times \mathcal{V}_N } \mathcal{L}_{\beta, Q}(u, f, z^k_{\beta, \theta})
    \\
    z^{k+1}_{\beta, \theta} &= z^k_{\beta, \theta} + \beta K(u^k_{\beta, \theta}, f^k_{\beta, \theta}).
  \end{split}
\end{equation}

In Figure \ref{fig: Augemented Lagrangian Poisson beta vary}, we show
the $\leb2$-error of the neural network approximation of the augmented
Lagrangian approach, for both state and control, against update number for
varying $\beta$ values and fixed $\alpha = 10^{-4}$.
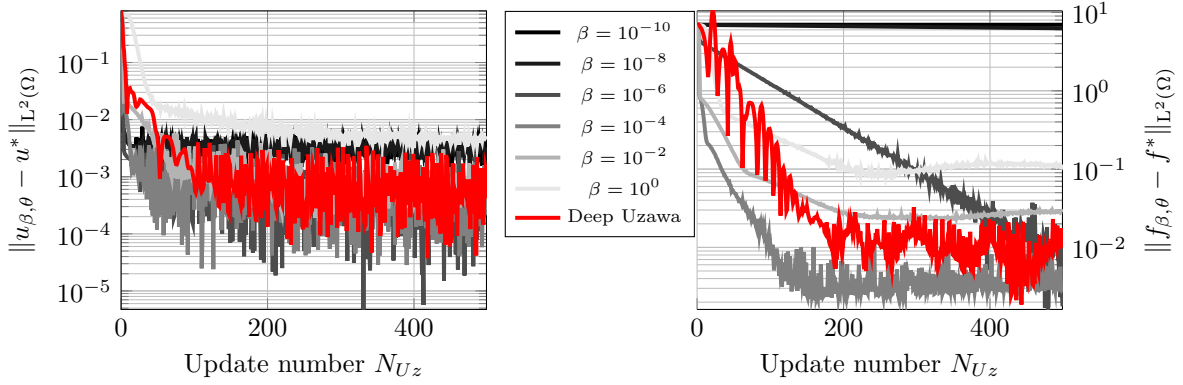
\begin{figure}[h!]
  %\ifthenelse{\boolean{showTikz}}
  {
  \begin{tikzpicture}

    \begin{groupplot}[
        group style={
            group size=2 by 1,
            horizontal sep=2.8cm,
            vertical sep=2cm,
        },
        enlargelimits=false,
        axis line style={draw=black}
    ]
    
    \nextgroupplot[
        width=0.39\textwidth,
        xlabel={Update number $N_{Uz}$},
        ylabel={$\| u_{\beta, \theta} - u^*\|_{\leb{2}(\W)}$},
        grid=both,
        ymode=log,
        log basis y={10},
        legend style={at={(1.05,1)}, anchor=north west, font=\scriptsize},
        axis line style={draw=black}
    ]
    
    \addplot[mark=none, color={black!110}, line width=1.5] table [col sep=comma, header=true, x expr=\coordindex, y index=0] {Plots/1D_Sine/Beta_vary/Vary_beta_10_nPts_201/Error.csv};
    \addlegendentry{$\beta=10^{-10}$}
    \addplot[mark=none, color={black!90}, line width=1.5] table [col sep=comma, header=true, x expr=\coordindex, y index=0] {Plots/1D_Sine/Beta_vary/Vary_beta_8_nPts_201/Error.csv};
    \addlegendentry{$\beta=10^{-8}$}
    \addplot[mark=none, color={black!70}, line width=1.5] table [col sep=comma, header=true, x expr=\coordindex, y index=0] {Plots/1D_Sine/Beta_vary/Vary_beta_6_nPts_201/Error.csv};
    \addlegendentry{$\beta=10^{-6}$}
    \addplot[mark=none, color={black!50}, line width=1.5] table [col sep=comma, header=true, x expr=\coordindex, y index=0] {Plots/1D_Sine/Beta_vary/Vary_beta_4_nPts_201/Error.csv};
    \addlegendentry{$\beta=10^{-4}$}
    \addplot[mark=none, color={black!30}, line width=1.5] table [col sep=comma, header=true, x expr=\coordindex, y index=0] {Plots/1D_Sine/Beta_vary/Vary_beta_2_nPts_201/Error.csv};
    \addlegendentry{$\beta=10^{-2}$}
    \addplot[mark=none, color={black!10}, line width=1.5] table [col sep=comma, header=true, x expr=\coordindex, y index=0] {Plots/1D_Sine/Beta_vary/Vary_beta_0_nPts_201/Error.csv};
    \addlegendentry{$\beta=10^{0}$}
    \addplot[mark=none, color={red}, line width=1.5] table [col sep=comma, header=true, x expr=\coordindex, y index=0] {Plots/1D_Sine/Vary_alpha/Vary_alpha_4_nPts_201/Error.csv};
    \addlegendentry{Deep Uzawa}
    
    \nextgroupplot[
        width=0.39\textwidth,
        xlabel={Update number $N_{Uz}$},
        ylabel={$\| f_{\beta, \theta} - f^* \|_{\leb{2}(\W)}$},
        grid=both,
        ymode=log,
        axis y line*=right,
        log basis y={10},
        legend style={at={(1.05,1)}, anchor=north west, font=\scriptsize},
        axis line style={draw=black}
    ]
    \addplot[mark=none, color={black!110}, line width=1.5] table [col sep=comma, header=true, x expr=\coordindex, y index=1] {Plots/1D_Sine/Beta_vary/Vary_beta_10_nPts_201/Error.csv};
    
    \addplot[mark=none, color={black!90}, line width=1.5] table [col sep=comma, header=true, x expr=\coordindex, y index=1] {Plots/1D_Sine/Beta_vary/Vary_beta_8_nPts_201/Error.csv};
    
    \addplot[mark=none, color={black!70}, line width=1.5] table [col sep=comma, header=true, x expr=\coordindex, y index=1] {Plots/1D_Sine/Beta_vary/Vary_beta_6_nPts_201/Error.csv};
    \addplot[mark=none, color={black!50}, line width=1.5] table [col sep=comma, header=true, x expr=\coordindex, y index=1] {Plots/1D_Sine/Beta_vary/Vary_beta_4_nPts_201/Error.csv};
    \addplot[mark=none, color={black!30}, line width=1.5] table [col sep=comma, header=true, x expr=\coordindex, y index=1] {Plots/1D_Sine/Beta_vary/Vary_beta_2_nPts_201/Error.csv};
    \addplot[mark=none, color={black!10}, line width=1.5] table [col sep=comma, header=true, x expr=\coordindex, y index=1] {Plots/1D_Sine/Beta_vary/Vary_beta_0_nPts_201/Error.csv};
    \addplot[mark=none, color={red}, line width=1.5] table [col sep=comma, header=true, x expr=\coordindex, y index=1] {Plots/1D_Sine/Vary_alpha/Vary_alpha_4_nPts_201/Error.csv};

    \end{groupplot}

\end{tikzpicture}
      }{
  }
  \caption{ \label{fig: Augemented Lagrangian Poisson beta vary}
    Example \ref{sec: 1D Sine function target}. lots of the
    $\leb2$-error of augmented Lagrangian approximation of state and
    control $u_{\beta,\theta}, f_{\beta,\theta} \in \mathcal{V}_N$,
    equation \eqref{eq: Augemented Lagrangian Poisson} for $\beta \in
    [10^{-10}, 1]$ and $\alpha = 10^{-4}$. The corresponding
    $\leb2$-error between the Deep Uzawa state and control is plotted
    in red. }
\end{figure}
Note that the augmented Lagrangian approach performs comparably to the
Deep Uzawa scheme for $\beta \approx 10^{-4}$.

\subsection{Example: Boundary Layers}\label{sec: 1D Boundary layer problem}

Again we take $\W := (0,1)$ but now $\mathcal{D} = 1$. As was observed
in \S \ref{sec:pdecon} the control can be eliminated from the KKT
system and the optimal state satisfies
\begin{equation}
  \begin{split}
    \alpha \Delta^2 u^* + u^* &= 1 \text{ in } \W
    \\
    u^* = \Delta u^* &= 0 \text{ on } \partial \W.
  \end{split}
\end{equation}
This problem induces boundary layers in the solution for small
$\alpha$. One may write a closed form solution to this problem; let
$\omega := (4\alpha)^{-1/4}$ and $y = \omega x$
\begin{equation}\label{eq:1D boundary layer exact solution}
    \begin{split}
        u^*(x) =&  1 
        - \cosh \bra{y }\cos {y} 
        + \dfrac{\sinh \bra{\omega}}{\cosh\bra{\omega} + \cos{\omega}} 
        \sinh \bra{y}
        \cos {y}
        - \dfrac{\sin{\omega}}{\cosh\bra{\omega} + \cos{\omega}} 
        \cosh \bra{y }
        \sin {y}         
        \\
        f^*(x) =& 
        2\omega^2\bra{\sinh \bra{y }
        \sin {y}
        -\dfrac{\sinh\bra{\omega}}{\cosh\bra{\omega} + \cos{\omega}} 
        \cosh \bra{y}
        \sin {y}
        - \dfrac{\sin{\omega}}{\cosh\bra{\omega} + \cos{\omega}} 
        \sinh \bra{y }
        \cos {y}}.
    \end{split}
\end{equation}
An approximation to these functions is given in Figure
\ref{fig:boundary_layer_state_control_with_exact}
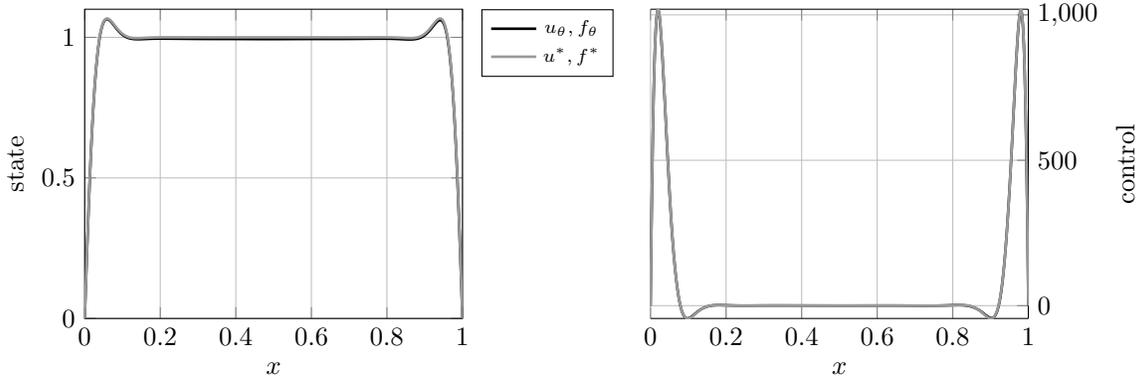
\begin{figure}[h!]
  %\ifthenelse{\boolean{showTikz}}
  {
    \begin{tikzpicture}
    \begin{groupplot}[
        group style={
            group size=2 by 1,
            horizontal sep=2.5cm,
            vertical sep=2cm,
        },
        enlargelimits=false,
        axis line style={draw=black}
    ] 
    \nextgroupplot[
        width=0.4\textwidth,
        xlabel={$x$},
        ylabel={state},
        xmin = 0,
        xmax = 1,
        ymin = 0,
        ymax = 1.1,
        grid=both,
        legend style={at={(1.05,1)}, anchor=north west, font=\scriptsize},
        axis line style={draw=black}
    ]    
    \addplot[mark=none, color={black!100}, line width=1.0] table [col sep=comma, header=true, x expr=\coordindex/500, y expr = \thisrowno{0}] {Plots/1D_boundary_layer/High_fid/Vary_alpha_7_nPts_501/State.csv};
    \addlegendentry{$u_{\theta}, f_{\theta}$}    
    \addplot[mark=none, color={black!40}, line width=1.0] table [col sep=comma, header=true, x expr=\coordindex/500, y expr = \thisrowno{0}] {Plots/1D_boundary_layer/High_fid/Vary_alpha_7_nPts_501/U_exact_alpha_7.csv};
    \addlegendentry{$u^*, f^*$}
    \nextgroupplot[
        width=0.4\textwidth,
        xlabel={$x$},
        ylabel={control},
        axis y line*=right,
        xmin = 0.,
        xmax =1.,
        grid=both,
        legend style={at={(1.05,1)}, anchor=north west, font=\scriptsize},
        axis line style={draw=black}
    ]
    \addplot[mark=none, color={black!100}, line width=1.0] table [col sep=comma, header=true, x expr=\coordindex/500, y expr = \thisrowno{0}] {Plots/1D_boundary_layer/High_fid/Vary_alpha_7_nPts_501/Control.csv};
    \addplot[mark=none, color={black!40}, line width=1.0] table [col sep=comma, header=true, x expr=\coordindex/500, y expr = \thisrowno{0}] {Plots/1D_boundary_layer/High_fid/Vary_alpha_7_nPts_501/F_exact_alpha_7.csv};
    \end{groupplot}
\end{tikzpicture}
  }{
  }
  \caption{\label{fig:boundary_layer_state_control_with_exact} Example
    \ref{sec: 1D Boundary layer problem}. Plots of the neural network
    approximation of the state and control and the exact solution, for
    $\alpha = 10^{-7}$.  }
\end{figure}

In Figure \ref{fig:boundary_layer_state_control_pointwise_error}, we
observe that the average of pointwise error between the neural network
approximations $u_\theta, f_\theta$ increases as $\alpha$
decreases. This is also observed in Figure
\ref{fig:1DBoundary_layer_err_v_epoch}, where we observe a smooth
decrease in the rate of convergence of both the state and control,
with respect to update number, as $\alpha$ decreases.
\begin{figure}[h!]
  %\ifthenelse{\boolean{showTikz}}
  {
  \begin{tikzpicture}
    \begin{groupplot}[
        group style={
            group size=2 by 1,
            horizontal sep=2.8cm,
            vertical sep=2cm,
        },
        enlargelimits=false,
        axis line style={draw=black}
    ]
 
    \nextgroupplot[
        width=0.39\textwidth,
        xlabel={$x$},
        ymode=log,
        ylabel = {$|u_{\theta}-u^*|$},
        log basis y={10},
        ymin = 1e-6,
        xmin = 0,
        xmax = 1,
        grid=both,
        legend style={at={(1.05,1)}, anchor=north west, font=\scriptsize}, 
        axis line style={draw=black}
    ]
    
    \addlegendentry{$\alpha = 1$}
    \addplot[mark=none, color={black!22}, line width=1.0] table [col sep=comma, header=true, x expr=\coordindex/500, y expr = \thisrowno{0}] {Plots/1D_boundary_layer/High_fid/Vary_alpha_0_nPts_501/State_pointwise.csv};
    \addlegendentry{$\alpha = 10^{-2}$}
    \addplot[mark=none, color={black!33}, line width=1.0] table [col sep=comma, header=true, x expr=\coordindex/500, y expr = \thisrowno{0}] {Plots/1D_boundary_layer/High_fid/Vary_alpha_2_nPts_501/State_pointwise.csv};
    \addlegendentry{$\alpha = 10^{-4}$}
    \addplot[mark=none, color={black!55}, line width=1.0] table [col sep=comma, header=true, x expr=\coordindex/500, y expr = \thisrowno{0}] {Plots/1D_boundary_layer/High_fid/Vary_alpha_4_nPts_501/State_pointwise.csv};
    \addlegendentry{$\alpha = 10^{-6}$}
    \addplot[mark=none, color={black!77}, line width=1.0] table [col sep=comma, header=true, x expr=\coordindex/500, y expr = \thisrowno{0}] {Plots/1D_boundary_layer/High_fid/Vary_alpha_6_nPts_501/State_pointwise.csv};
    \addlegendentry{$\alpha = 10^{-8}$}
    \addplot[mark=none, color={black!99}, line width=1.0] table [col sep=comma, header=true, x expr=\coordindex/500, y expr = \thisrowno{0}] {Plots/1D_boundary_layer/High_fid/Vary_alpha_8_nPts_501/State_pointwise.csv};
    
    \nextgroupplot[
      width=0.39\textwidth,
        xlabel={$x$},
        ylabel = {$|f_{\theta}-f^*|/\|f^*\|_{\leb2{(\W)}}$},
        axis y line*=right,
        ymode=log,
        log basis y={10},
        ymin = 1e-7,
        xmin = 0,
        xmax = 1,
        grid=both,
        legend style={at={(1.05,1)}, anchor=north west, font=\scriptsize},
        axis line style={draw=black}
    ]
    
    \addplot[mark=none, color={black!22}, line width=1.0] table [col sep=comma, header=true, x expr=\coordindex/500, y expr = \thisrowno{0}/0.127789462408935] {Plots/1D_boundary_layer/High_fid/Vary_alpha_0_nPts_501/Control_pointwise.csv};
    \addplot[mark=none, color={black!33}, line width=1.0] table [col sep=comma, header=true, x expr=\coordindex/500, y expr = \thisrowno{0}/6.38409025200079] {Plots/1D_boundary_layer/High_fid/Vary_alpha_2_nPts_501/Control_pointwise.csv};
    \addplot[mark=none, color={black!55}, line width=1.0] table [col sep=comma, header=true, x expr=\coordindex/500, y expr = \thisrowno{0}/26.3652888824987] {Plots/1D_boundary_layer/High_fid/Vary_alpha_4_nPts_501/Control_pointwise.csv};
    \addplot[mark=none, color={black!77}, line width=1.0] table [col sep=comma, header=true, x expr=\coordindex/500, y expr = \thisrowno{0}/149.534879351525] {Plots/1D_boundary_layer/High_fid/Vary_alpha_6_nPts_501/Control_pointwise.csv};
    \addplot[mark=none, color={black!99}, line width=1.0] table [col sep=comma, header=true, x expr=\coordindex/500, y expr = \thisrowno{0}/840.896415253715] {Plots/1D_boundary_layer/High_fid/Vary_alpha_8_nPts_501/Control_pointwise.csv};

    \end{groupplot}
\end{tikzpicture}
  }{
  }
\caption{\label{fig:boundary_layer_state_control_pointwise_error}
  Example \ref{sec: 1D Boundary layer problem}. Plots of the pointwise-error of the neural network approximation of the state and control, for $\alpha \in [10^{-8},1]$. 
  }
\end{figure}

In Figure \ref{fig:1DBoundary_layer_err_v_NSGD_constant_LR}, we plot
the $\leb2$-error of the neural network approximation to the state and
control, against the update number $N_{\text{Uz}}$, and vary the
number of stochastic gradient descent iterations $N_{\text{SDG}}$. In
this figure we observe that for small $N_{\text{SDG}}$, the error
fails to converge with respect to the update number; for large
$N_{\text{SDG}}$, the error for both the state and control converges.
\begin{figure}
  \centering
  %\ifthenelse{\boolean{showTikz}}
      {
        \begin{tikzpicture}

    \begin{groupplot}[
        group style={
            group size=2 by 1,
            horizontal sep=2.6cm, 
            vertical sep=2cm,
        },
        enlargelimits=false,
        axis line style={draw=black} 
    ]

      \nextgroupplot[
        width=0.39\textwidth,
        xlabel={Update number $N_{Uz}$},
        ylabel={$\| u_\theta - u\|_{\leb{2}(\W)}$},
        grid=both,
        xmin = 0,
        ymode=log,
        log basis y={10},
        legend style={at={(1.05,1)}, anchor=north west, font=\scriptsize}, 
        axis line style={draw=black}
    ]
    \addplot[mark=none, color={black!11}, line width=1.5] table [col sep=comma, header=true, x expr=\coordindex, y expr = \thisrowno{0}] {Plots/1D_boundary_layer/High_fid/Vary_alpha_0_nPts_501/Error.csv};
    \addlegendentry{$\alpha=1$}
    \addplot[mark=none, color={black!33}, line width=1.5] table [col sep=comma, header=true, x expr=\coordindex, y expr = \thisrowno{0}] {Plots/1D_boundary_layer/High_fid/Vary_alpha_2_nPts_501/Error.csv};
    \addlegendentry{$\alpha=10^{-2}$}
    \addplot[mark=none, color={black!55}, line width=1.5] table [col sep=comma, header=true, x expr=\coordindex, y expr = \thisrowno{0}] {Plots/1D_boundary_layer/High_fid/Vary_alpha_4_nPts_501/Error.csv};
    \addlegendentry{$\alpha=10^{-4}$}
    \addplot[mark=none, color={black!77}, line width=1.5] table [col sep=comma, header=true, x expr=\coordindex, y expr = \thisrowno{0}] {Plots/1D_boundary_layer/High_fid/Vary_alpha_6_nPts_501/Error.csv};
    \addlegendentry{$\alpha=10^{-6}$}
    \addplot[mark=none, color={black!99}, line width=1.5] table [col sep=comma, header=true, x expr=\coordindex, y expr = \thisrowno{0}] {Plots/1D_boundary_layer/High_fid/Vary_alpha_8_nPts_501/Error.csv};
    \addlegendentry{$\alpha=10^{-8}$}
      \nextgroupplot[
        width=0.39\textwidth,
        xlabel={Update number $N_{Uz}$},
        ylabel = {$\Norm{f_{\theta}-f^*}_{\leb{2}(\W)}/\|f^*\|_{\leb{2(\W)}}$},        
        grid=both,
        ymode=log,
        xmin = 0,
        axis y line*=right,
        log basis y={10},
        legend style={at={(1.05,1)}, anchor=north west},
        axis line style={draw=black} 
    ]
    \addplot[mark=none, color={black!22}, line width=1.5] table [col sep=comma, header=true, x expr=\coordindex, y expr = \thisrowno{1}/0.127789462408935] {Plots/1D_boundary_layer/High_fid/Vary_alpha_0_nPts_501/Error.csv};
    \addplot[mark=none, color={black!33}, line width=1.5] table [col sep=comma, header=true, x expr=\coordindex, y expr = \thisrowno{1}/6.38409025200079] {Plots/1D_boundary_layer/High_fid/Vary_alpha_2_nPts_501/Error.csv};
    \addplot[mark=none, color={black!55}, line width=1.5] table [col sep=comma, header=true, x expr=\coordindex, y expr = \thisrowno{1}/26.3652888824987] {Plots/1D_boundary_layer/High_fid/Vary_alpha_4_nPts_501/Error.csv};
    \addplot[mark=none, color={black!77}, line width=1.5] table [col sep=comma, header=true, x expr=\coordindex, y expr = \thisrowno{1}/149.534879351525] {Plots/1D_boundary_layer/High_fid/Vary_alpha_6_nPts_501/Error.csv};
    \addplot[mark=none, color={black!99}, line width=1.5] table [col sep=comma, header=true, x expr=\coordindex, y expr = \thisrowno{1}/840.896415253715] {Plots/1D_boundary_layer/High_fid/Vary_alpha_8_nPts_501/Error.csv};
    \end{groupplot}

\end{tikzpicture}
      }{
      }
      \caption{Example \ref{sec: 1D Boundary layer problem}.  Plots of
        the $\leb{2}$-error of the neural network approximation of the
        state and control, for $\alpha \in [10^{-8},1]$.}
\label{fig:1DBoundary_layer_err_v_epoch}
\end{figure}

\begin{figure}
  \centering
  %  \ifthenelse{\boolean{showTikz}}
      {
        \begin{tikzpicture}

    \begin{groupplot}[
        group style={
            group size=2 by 1,
            horizontal sep=2.8cm, 
            vertical sep=2cm,
        },
        enlargelimits=false,
        axis line style={draw=black} 
    ]

      \nextgroupplot[
        width=0.39\textwidth,
        xlabel={Update number $N_{\text{Uz}}$},
        ylabel={$\| u_\theta - u^* \|_{L^2(\Omega)}$},
        grid=both,
        ymode=log,
        log basis y={10},
        legend style={at={(1.02,1)}, anchor=north west, font=\scriptsize},
        axis line style={draw=black}
    ]
    \addplot[mark=none, color={black!7}, line width=1.5] table [col sep=comma, header=true, x expr=\coordindex, y expr = \thisrowno{0}] {Plots/1D_boundary_layer/Constant_LR/Vary_alpha_5_nPts_201_EpUp_1/Error.csv};
    \addlegendentry{$N_{\text{SGD}}=1$}
    \addplot[mark=none, color={black!35}, line width=1.5] table [col sep=comma, header=true, x expr=\coordindex, y expr = \thisrowno{0}] {Plots/1D_boundary_layer/Constant_LR/Vary_alpha_5_nPts_201_EpUp_5/Error.csv};
    \addlegendentry{$N_{\text{SGD}}=5$}
    \addplot[mark=none, color={black!71}, line width=1.5] table [col sep=comma, header=true, x expr=\coordindex, y expr = \thisrowno{0}] {Plots/1D_boundary_layer/Constant_LR/Vary_alpha_5_nPts_201_EpUp_10/Error.csv};
    \addlegendentry{$N_{\text{SGD}}=10$}
    \addplot[mark=none, color={black!78}, line width=1.5] table [col sep=comma, header=true, x expr=\coordindex, y expr = \thisrowno{0}] {Plots/1D_boundary_layer/Constant_LR/Vary_alpha_5_nPts_201_EpUp_40/Error.csv};
    \addlegendentry{$N_{\text{SGD}}=40$}
    \addplot[mark=none, color={black!92}, line width=1.5] table [col sep=comma, header=true, x expr=\coordindex, y expr = \thisrowno{0}] {Plots/1D_boundary_layer/Constant_LR/Vary_alpha_5_nPts_201_EpUp_100/Error.csv};
    \addlegendentry{$N_{\text{SGD}}=100$}
      \nextgroupplot[
        width=0.39\textwidth,
        xlabel={Update number $N_{\text{Uz}}$},
        ylabel = {$\Norm{f_{\theta}-f^*}_{\leb{2}(\W)}/\Norm{f^*}_{\leb{2(\W)}}$},        
        grid=both,
        ymode=log,
        xmin = 0,
        xmax = 500,
        axis y line*=right,
        log basis y={10},
        legend style={at={(1.05,1)}, anchor=north west},
        axis line style={draw=black} 
    ]
    \addplot[mark=none, color={black!7}, line width=1.5] table [col sep=comma, header=true, x expr=\coordindex, y expr = \thisrowno{1}/26.3652888824987] {Plots/1D_boundary_layer/Constant_LR/Vary_alpha_5_nPts_201_EpUp_1/Error.csv};
    \addplot[mark=none, color={black!35}, line width=1.5] table [col sep=comma, header=true, x expr=\coordindex, y expr = \thisrowno{1}/26.3652888824987] {Plots/1D_boundary_layer/Constant_LR/Vary_alpha_5_nPts_201_EpUp_5/Error.csv};
    \addplot[mark=none, color={black!71}, line width=1.5] table [col sep=comma, header=true, x expr=\coordindex, y expr = \thisrowno{1}/26.3652888824987] {Plots/1D_boundary_layer/Constant_LR/Vary_alpha_5_nPts_201_EpUp_10/Error.csv};
    \addplot[mark=none, color={black!78}, line width=1.5] table [col sep=comma, header=true, x expr=\coordindex, y expr = \thisrowno{1}/26.3652888824987] {Plots/1D_boundary_layer/Constant_LR/Vary_alpha_5_nPts_201_EpUp_40/Error.csv};
    \addplot[mark=none, color={black!92}, line width=1.5] table [col sep=comma, header=true, x expr=\coordindex, y expr = \thisrowno{1}/26.3652888824987] {Plots/1D_boundary_layer/Constant_LR/Vary_alpha_5_nPts_201_EpUp_100/Error.csv};

    \end{groupplot}

\end{tikzpicture}
      }{
      }
      \caption{Example \ref{sec: 1D Boundary layer problem}. Plots of
        the $\leb2$-error of the neural network approximation of the
        state and control against update number $N_{\text{Uz}}$ for
        $\alpha = 10^{-5}$, and $N_{\text{SGD}} \in [1,100]$.}
      \label{fig:1DBoundary_layer_err_v_NSGD_constant_LR}
\end{figure}
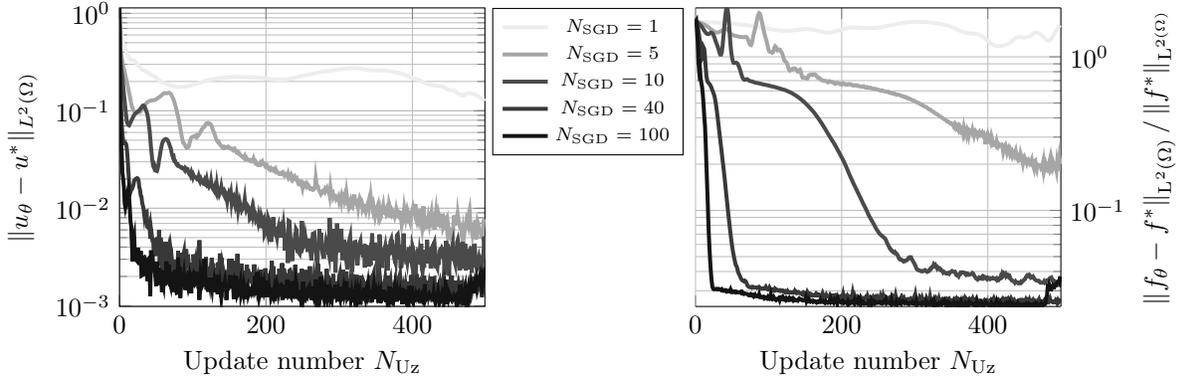

\clearpage

\subsection{Example: 2d Smooth Target}
\label{sec: 2D Sine function target}
Consider the domain $\W := (0,1)^2$ and the target function
$\mathcal{D}$ is given by
\begin{equation}
  \mathcal{D}(x) = \bra{1+ 4 \alpha \pi^4}\sin{\pi x}\sin{\pi y}. \label{eq: 2D sine function}
\end{equation}  
The optimal state and control is then 
\begin{equation}\label{eq: Exact solution 2D sine}
  u^*(x,y) = \sin{\pi x}\sin{\pi y}, \qquad f^*(x,y) = 2\pi^2 \sin{\pi x}\sin{\pi y}.
\end{equation}
In Figure \ref{fig:2Dsinesine_plot}, we show a neural network
approximation to the state and control as well as the pointwise
error. Additionally, we have that for that control, the points of
greatest error occur on the boundary.

In Figure \ref{fig:2D_sine_sine_layer_err_v_epoch}, we examine the
$\leb{2}(\W)$ errors for both state and control for various $\alpha$.
\begin{figure}[hbtp]
    \begin{subfigure}[t]{0.48\textwidth}
        \centering
        \includegraphics[width=\textwidth]{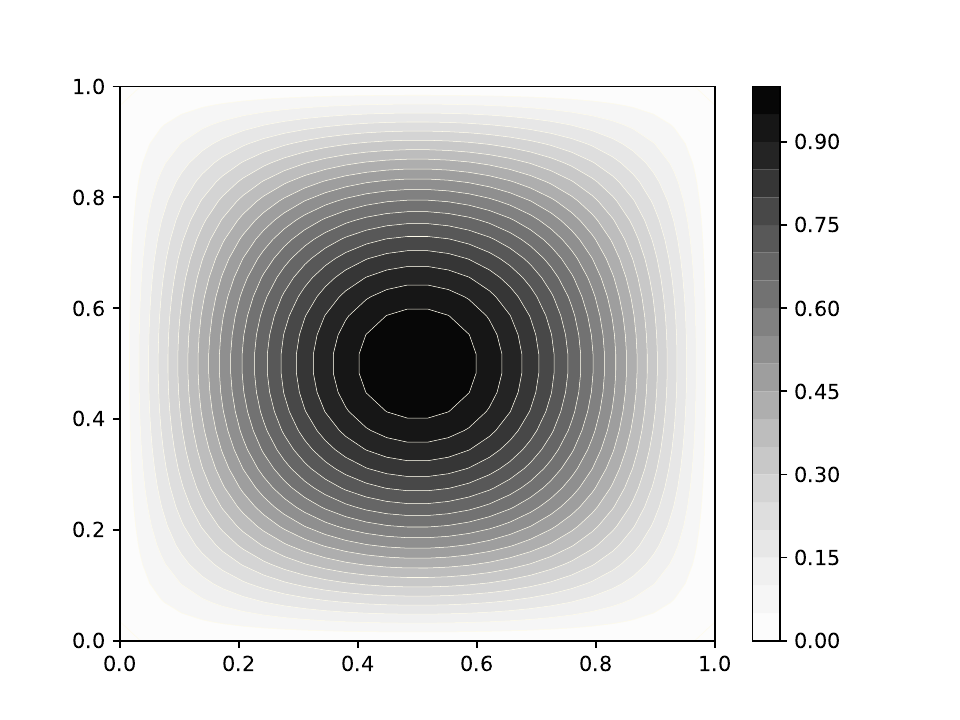}
        \caption{Plot of $u_{\theta}$.}
        \label{subfig:2D_sine_pointwise_state}
    \end{subfigure}
    \hfill
    \begin{subfigure}[t]{0.48\textwidth}
        \centering
        \includegraphics[width=\textwidth]{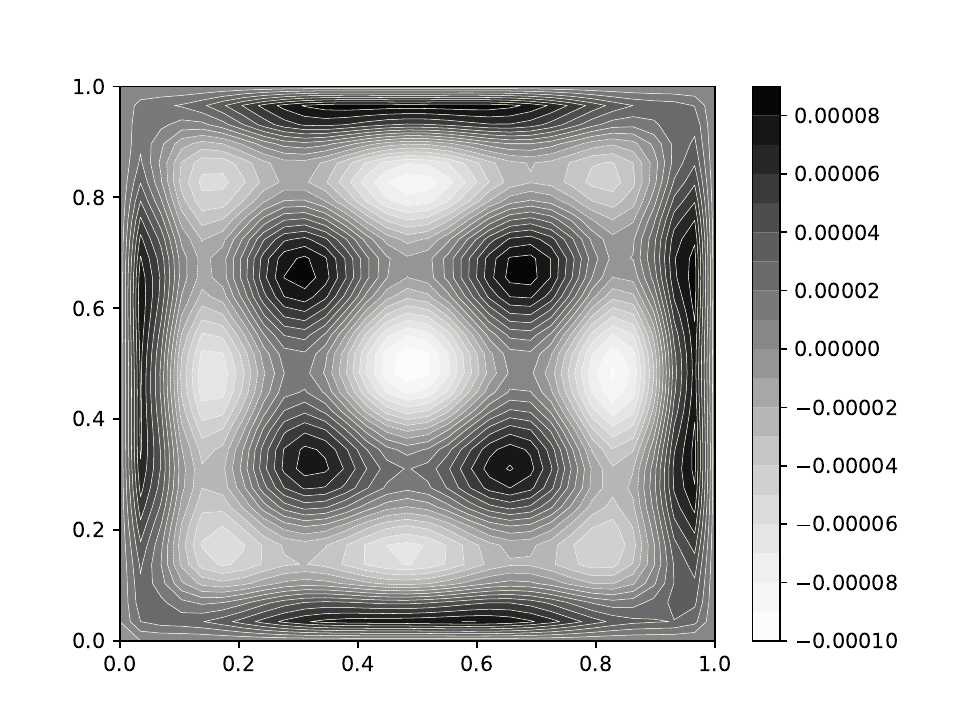}
        \caption{Pointwise difference between $u_{\theta}$ and $u^*$.}
        \label{subfig:2D_sine_pointwise_state_error}
    \end{subfigure}
\\
    \begin{subfigure}[t]{0.48\textwidth}
        \centering
        \includegraphics[width=\textwidth]{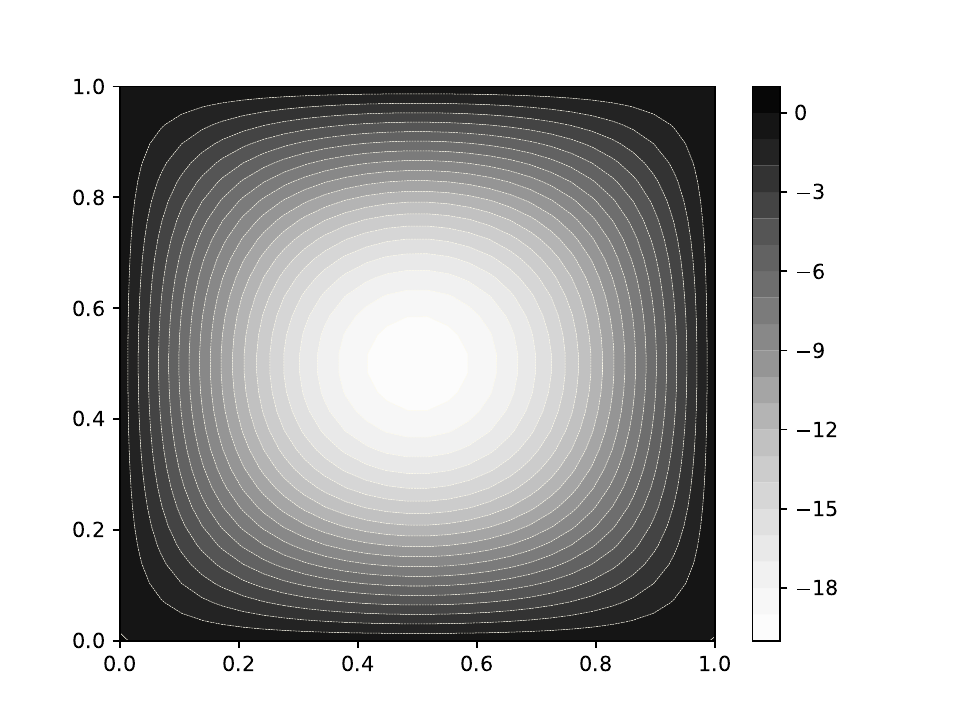}
        \caption{Plot of $f_{\theta}$.}
        \label{subfig:2D_sine_pointwise_control}
    \end{subfigure}
    \hfill
    \begin{subfigure}[t]{0.48\textwidth}
        \centering
        \includegraphics[width=\textwidth]{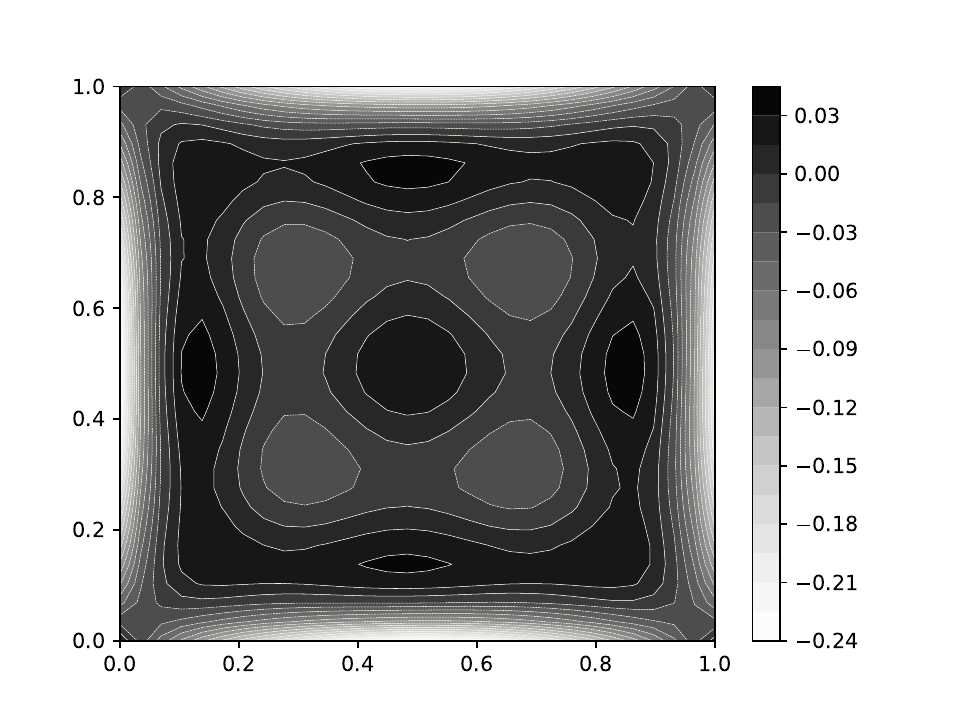}
        \caption{Pointwise difference between $f_{\theta}$ and $f^*$.}
        \label{subfig:2D_sine_pointwise_control_error}
    \end{subfigure}
    \caption{Example \ref{sec: 2D Sine function target}. Plots of $u_{\theta}$ and $f_{\theta}$ (left), and $u_{\theta}-u^*$ and $f_{\theta}-f^*$ (right);
      for $\alpha = 10^{-4}$ and $N = 30$. \label{fig:2Dsinesine_plot}
    }
\end{figure}
\begin{figure}
    \centering
    %\ifthenelse{\boolean{showTikz}}
        {
          \begin{tikzpicture}

    \begin{groupplot}[
        group style={
            group size=2 by 1,
            horizontal sep=2.8cm,
            vertical sep=2cm,
        },
        enlargelimits=false,
        axis line style={draw=black}
    ]

      \nextgroupplot[
        width=0.39\textwidth,
        xlabel={Update number $N_{\rm{Uz}}$},
        ylabel={$\| u - u_\theta \|_{\leb{2}(\W)}$},
        grid=both,
        ymode=log,
        xmin = 0,
        log basis y={10},
        legend style={at={(1.05,1)}, anchor=north west, font=\scriptsize}, 
        axis line style={draw=black} 
    ]
    \addplot[mark=none, color={black!20}, line width=1.5] table [col sep=comma, header=true, x expr=\coordindex, y index=0] {Plots/2D_Sine_Sine/Vary_Alpha/Vary_alpha_0_nPts_30/Error.csv};
    \addlegendentry{$\alpha=1$}
    \addplot[mark=none, color={black!30}, line width=1.5] table [col sep=comma, header=true, x expr=\coordindex, y index=0] {Plots/2D_Sine_Sine/Vary_Alpha/Vary_alpha_2_nPts_30/Error.csv};
    \addlegendentry{$\alpha=10^{-2}$}
    \addplot[mark=none, color={black!50}, line width=1.5] table [col sep=comma, header=true, x expr=\coordindex, y index=0] {Plots/2D_Sine_Sine/Vary_Alpha/Vary_alpha_4_nPts_30/Error.csv};
    \addlegendentry{$\alpha=10^{-4}$}
    \addplot[mark=none, color={black!70}, line width=1.5] table [col sep=comma, header=true, x expr=\coordindex, y index=0] {Plots/2D_Sine_Sine/Vary_Alpha/Vary_alpha_6_nPts_30/Error.csv};
    \addlegendentry{$\alpha=10^{-6}$}
    \addplot[mark=none, color={black!90}, line width=1.5] table [col sep=comma, header=true, x expr=\coordindex, y index=0] {Plots/2D_Sine_Sine/Vary_Alpha/Vary_alpha_8_nPts_30/Error.csv};
    \addlegendentry{$\alpha=10^{-8}$}
      \nextgroupplot[
        width=0.39\textwidth,
        xlabel={Update number $N_{\rm{Uz}}$},
        ylabel={$\| f - f_\theta \|_{\leb{2}(\W)}$},
        grid=both,
        ymode=log,
        xmin = 0,
        axis y line*=right,
        log basis y={10},
        legend style={at={(1.05,1)}, anchor=north west},
        axis line style={draw=black} 
    ]
    \addplot[mark=none, color={black!20}, line width=1.5] table [col sep=comma, header=true, x expr=\coordindex, y index=1] {Plots/2D_Sine_Sine/Vary_Alpha/Vary_alpha_0_nPts_30/Error.csv};
    \addplot[mark=none, color={black!30}, line width=1.5] table [col sep=comma, header=true, x expr=\coordindex, y index=1] {Plots/2D_Sine_Sine/Vary_Alpha/Vary_alpha_2_nPts_30/Error.csv};
    \addplot[mark=none, color={black!50}, line width=1.5] table [col sep=comma, header=true, x expr=\coordindex, y index=1] {Plots/2D_Sine_Sine/Vary_Alpha/Vary_alpha_4_nPts_30/Error.csv};
    \addplot[mark=none, color={black!70}, line width=1.5] table [col sep=comma, header=true, x expr=\coordindex, y index=1] {Plots/2D_Sine_Sine/Vary_Alpha/Vary_alpha_6_nPts_30/Error.csv};
    \addplot[mark=none, color={black!90}, line width=1.5] table [col sep=comma, header=true, x expr=\coordindex, y index=1] {Plots/2D_Sine_Sine/Vary_Alpha/Vary_alpha_8_nPts_30/Error.csv};
    \end{groupplot}

\end{tikzpicture}
        }{
  }
\caption{Example \ref{sec: 2D Sine function target}. Plots of the $\leb2$-error between $u_\theta, f_\theta \in \mathcal{V}_N$, and $u^*, f^* \in \mathcal{V}_N$, from equation \eqref{eq: Exact solution 2D sine}; for $\alpha \in [10^{-8},1]$.
} \label{fig:2D_sine_sine_layer_err_v_epoch}
\end{figure}
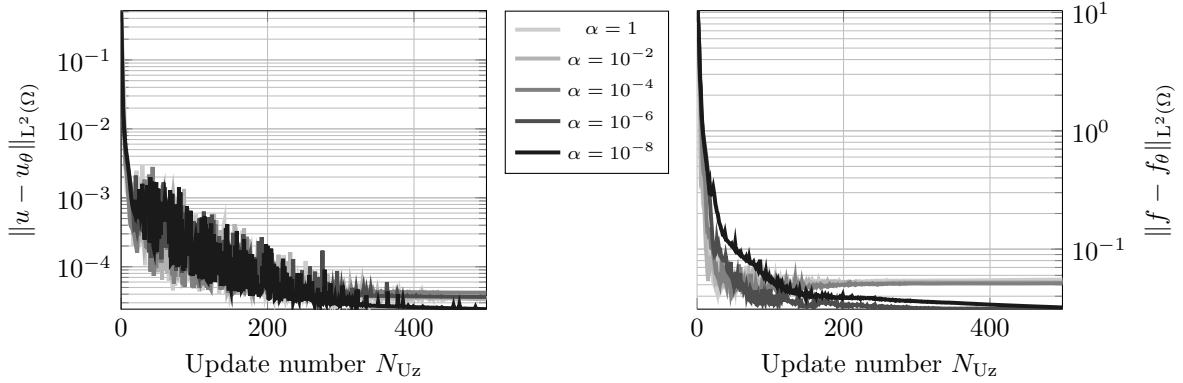

\section{Numerical Experiments: Allen-Cahn equation}\label{sec:Allen-Cahn}

To demonstrate robustness of the methodology we consider a semi-linear
PDE constraint, the Allen-Cahn equation. Let $\epsilon > 0$ be given,
we seek a $u^* \in \widehat{\sobh{}} := \bra{\leb6 \cap \sobhz{1} \cap
  \sobh{2}}(\W)$ and $f^* \in \leb2(\W)$ such that
\begin{equation}
  u^*,f^*
  :=
  \argmin\limits_{u,f \in \widehat{\sobh{}}\times \leb2(\W)}
  \bra{\dfrac{1}{2}\Norm{u-\mathcal{D}}_{\leb2(\W)}^2
    +
    \dfrac{\alpha}{4}\Norm{f}_{\leb2(\W)}^2
    +
    \dfrac{\alpha}{4}\Norm{Au}_{\leb2(\W)}^2}
\end{equation}
subject to the constraint
\begin{equation}\label{eq: Allen-Cahn equation}
  A u := -\Delta u - \frac 1 {\epsilon^2} u \qp{1 - u^2} = f.
\end{equation}
This is a challenging operator which is not guaranteed to have a
unique solution.

We compute the linearisation of the operator as
\begin{equation}
  \D A \phi
  :=
  - \Delta \phi
  -
  \frac 1{\epsilon^{2}} \qp{
    \phi
    +
    3 \phi u^2 },
\end{equation}
then the first-order optimality conditions are 
\begin{equation}
  \begin{split}
    \ip{Au - f}{\phi}_{\leb2(\W)} &= 0 \qquad \forall \phi \in \leb{2}(\W)
    \\
    \ip{u - \mathcal{D}}{\psi}_{\leb2(\W)}
    +
    \dfrac{\alpha}{2}\ip{Au}{\D A \psi}_{\leb2(\W)}
    -
    \ip{z}{\D A\psi}_{\leb2(\W)} &=0 \qquad \forall \psi \in \widehat{\sobh{}} \\
    \ip{\alpha f - z}{\xi}_{\leb2(\W)} &=0 \qquad \forall \xi \in \leb2(\W).
  \end{split}
\end{equation}

\subsection{Example: Sine solution}\label{Sec: Allen-Cahn equation}
We choose target function $\mathcal{D}$ such that an exact solution is
given by
\begin{equation}\label{eq: Allen-Cahn state and control exact}
  u^*(x) := \sin {\pi x}, \qquad
  f^*(x) := \sin{\pi x}\bra{\pi^2 - \dfrac{1}{\epsilon^2}\bra{\cos{\pi x}}^2 }.
\end{equation}
For a fixed $\epsilon$, we observe in Figure \ref{fig:Allen-Cahn Sine1
  pointwise error} that as $\alpha$ decreases, the pointwise error
between the neural network approximation and the exact solution
increases. We also observe similar behaviour when we consider the
$\leb2(\W)$ error in Figure \ref{fig:Allen-Cahn Sine1 L2 error},
where the rate of convergence of the control, with respect to update number,
decreases as $\alpha$ decreases. For both the pointwise error and the
$\leb2(\W)$ error, the state is robust with respect to $\alpha$.

For a fixed $\alpha$ value, we observe in Figure
\ref{fig:AllenCahnepsvary} that the neural network approximations
$u_\theta,f_\theta$ fail to converge to the solutions $u^*,f^*$, with
respect to update number, as $\epsilon$ decreases. This is a manifestation of
the lack of uniqueness of solution to the constraint, which may be observed in Figure \ref{fig:Allen-Cahn Sine1 state and control}.

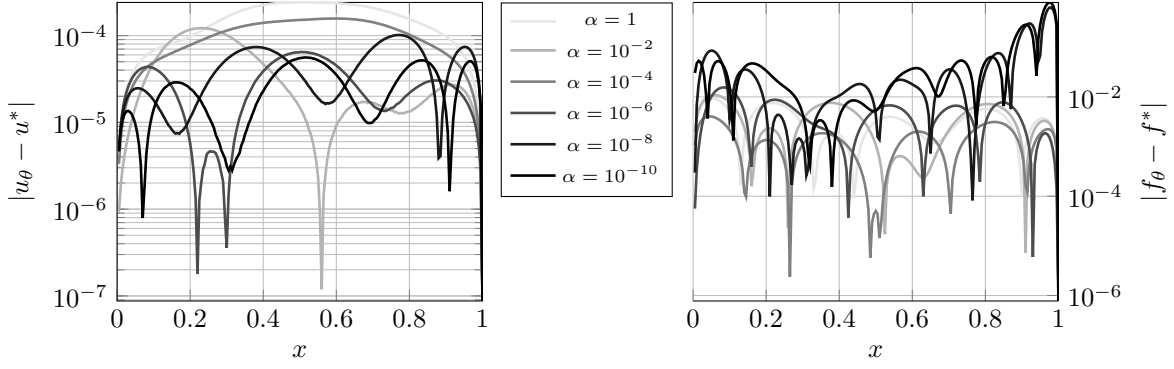
\begin{figure}[hbtp]
    \centering
    %\ifthenelse{\boolean{showTikz}}
        {
          \begin{tikzpicture}
    \begin{groupplot}[
        group style={
            group size=2 by 1,
            horizontal sep=2.8cm, 
            vertical sep=2cm,
        },
        enlargelimits=false,
        axis line style={draw=black} 
    ]
      \nextgroupplot[
        width=0.39\textwidth,
        xlabel={$x$},
        ylabel= {$|u_{\theta}-u^*|$},
        grid=both,
        ymode=log,
        xmin = 0,
        xmax = 1,
        log basis y={10},
        legend style={at={(1.05,1)}, anchor=north west, font=\scriptsize},
        axis line style={draw=black}
    ]
    \addlegendentry{$\alpha = 1$}
    \addplot[mark=none, color={black!10}, line width=1.0] table [col sep=comma, header=true, x expr=\coordindex/200, y index=0] {Plots/1D_Allen_Cahn_alpha_vary/Vary_alpha_0_nPts_201/Pointwise_State.csv};
    \addlegendentry{$\alpha = 10^{-2}$}
    \addplot[mark=none, color={black!30}, line width=1.0] table [col sep=comma, header=true, x expr=\coordindex/200, y index=0] {Plots/1D_Allen_Cahn_alpha_vary/Vary_alpha_2_nPts_201/Pointwise_State.csv};
    \addlegendentry{$\alpha = 10^{-4}$}
    \addplot[mark=none, color={black!50}, line width=1.0] table [col sep=comma, header=true, x expr=\coordindex/200, y index=0] {Plots/1D_Allen_Cahn_alpha_vary/Vary_alpha_4_nPts_201/Pointwise_State.csv};
    \addlegendentry{$\alpha = 10^{-6}$}
    \addplot[mark=none, color={black!70}, line width=1.0] table [col sep=comma, header=true, x expr=\coordindex/200, y index=0] {Plots/1D_Allen_Cahn_alpha_vary/Vary_alpha_6_nPts_201/Pointwise_State.csv};
    \addlegendentry{$\alpha = 10^{-8}$}
    \addplot[mark=none, color={black!90}, line width=1.0] table [col sep=comma, header=true, x expr=\coordindex/200, y index=0] {Plots/1D_Allen_Cahn_alpha_vary/Vary_alpha_8_nPts_201/Pointwise_State.csv};
    \addlegendentry{$\alpha = 10^{-10}$}
    \addplot[mark=none, color={black!110}, line width=1.0] table [col sep=comma, header=true, x expr=\coordindex/200, y index=0] {Plots/1D_Allen_Cahn_alpha_vary/Vary_alpha_10_nPts_201/Pointwise_State.csv};
      \nextgroupplot[
        width=0.39\textwidth,
        ylabel= {$|f_{\theta}-f^*|$},
        xlabel={$x$},
        axis y line*=right,
        xmin = 0,
        xmax = 1,
        grid=both,
        ymode=log,
        log basis y={10},
        legend style={at={(1.05,1)}, anchor=north west, font=\scriptsize},
        axis line style={draw=black} 
    ]
    \addplot[mark=none, color={black!10}, line width=1.0] table [col sep=comma, header=true, x expr=\coordindex/200, y index=0] {Plots/1D_Allen_Cahn_alpha_vary/Vary_alpha_0_nPts_201/Pointwise_Control.csv};
    \addplot[mark=none, color={black!30}, line width=1.0] table [col sep=comma, header=true, x expr=\coordindex/200, y index=0] {Plots/1D_Allen_Cahn_alpha_vary/Vary_alpha_2_nPts_201/Pointwise_Control.csv};
    \addplot[mark=none, color={black!50}, line width=1.0] table [col sep=comma, header=true, x expr=\coordindex/200, y index=0] {Plots/1D_Allen_Cahn_alpha_vary/Vary_alpha_4_nPts_201/Pointwise_Control.csv};
    \addplot[mark=none, color={black!70}, line width=1.0] table [col sep=comma, header=true, x expr=\coordindex/200, y index=0] {Plots/1D_Allen_Cahn_alpha_vary/Vary_alpha_6_nPts_201/Pointwise_Control.csv};
    \addplot[mark=none, color={black!90}, line width=1.0] table [col sep=comma, header=true, x expr=\coordindex/200, y index=0] {Plots/1D_Allen_Cahn_alpha_vary/Vary_alpha_8_nPts_201/Pointwise_Control.csv};
    \addplot[mark=none, color={black!110}, line width=1.0] table [col sep=comma, header=true, x expr=\coordindex/200, y index=0] {Plots/1D_Allen_Cahn_alpha_vary/Vary_alpha_10_nPts_201/Pointwise_Control.csv};
    \end{groupplot}
\end{tikzpicture}
        }{
        }
        \caption{Example \ref{Sec: Allen-Cahn equation}. Plots of the
          pointwise error between the neural network approximate state
          and control and the exact solution, for $\epsilon = 1$ and
          varying $\alpha$ values. }
    \label{fig:Allen-Cahn Sine1 pointwise error}
\end{figure}

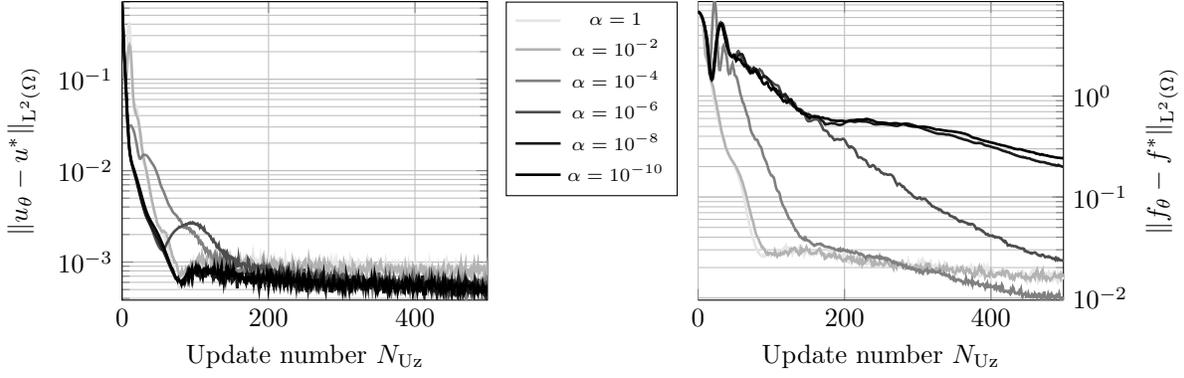
\begin{figure}[hbtp]
    \centering
    %\ifthenelse{\boolean{showTikz}}
        {
          \begin{tikzpicture}
    \begin{groupplot}[
        group style={
            group size=2 by 1,
            horizontal sep=2.8cm, 
            vertical sep=2cm,
        },
        enlargelimits=false,
        axis line style={draw=black}
    ]
      \nextgroupplot[
        width=0.39\textwidth,
        ylabel= {$\Norm{ u_{\theta}-u^*}_{\leb{2}(\W)}$},
        xlabel={Update number $N_{\rm{Uz}}$},
        xmin = 0,
        grid=both,
        ymode=log,
        log basis y={10},
        legend style={at={(1.05,1)}, anchor=north west, font=\scriptsize},
        axis line style={draw=black} 
    ]
    \addlegendentry{$\alpha = 1$}
    \addplot[mark=none, color={black!10}, line width=1.0] table [col sep=comma, header=true, x expr=\coordindex, y index=0] {Plots/1D_Allen_Cahn_alpha_vary/Vary_alpha_0_nPts_201/Error.csv};
    \addlegendentry{$\alpha = 10^{-2}$}
    \addplot[mark=none, color={black!30}, line width=1.0] table [col sep=comma, header=true, x expr=\coordindex, y index=0] {Plots/1D_Allen_Cahn_alpha_vary/Vary_alpha_2_nPts_201/Error.csv};
    \addlegendentry{$\alpha = 10^{-4}$}
    \addplot[mark=none, color={black!50}, line width=1.0] table [col sep=comma, header=true, x expr=\coordindex, y index=0] {Plots/1D_Allen_Cahn_alpha_vary/Vary_alpha_4_nPts_201/Error.csv};
    \addlegendentry{$\alpha = 10^{-6}$}
    \addplot[mark=none, color={black!70}, line width=1.0] table [col sep=comma, header=true, x expr=\coordindex, y index=0] {Plots/1D_Allen_Cahn_alpha_vary/Vary_alpha_6_nPts_201/Error.csv};
    \addlegendentry{$\alpha = 10^{-8}$}
    \addplot[mark=none, color={black!90}, line width=1.0] table [col sep=comma, header=true, x expr=\coordindex, y index=0] {Plots/1D_Allen_Cahn_alpha_vary/Vary_alpha_8_nPts_201/Error.csv};
    \addlegendentry{$\alpha = 10^{-10}$}
    \addplot[mark=none, color={black!110}, line width=1.0] table [col sep=comma, header=true, x expr=\coordindex, y index=0] {Plots/1D_Allen_Cahn_alpha_vary/Vary_alpha_10_nPts_201/Error.csv};
% Left plot
    \nextgroupplot[
      width=0.39\textwidth,
      ylabel= {$\Norm{ f_{\theta}-f^*}_{\leb{2}(\W)}$},
      xlabel={Update number $N_{\rm{Uz}}$},
      axis y line*=right,
      grid=both,
      xmin = 0,
        ymode=log,
        log basis y={10},
        legend style={at={(1.05,1)}, anchor=north west, font=\scriptsize},
        axis line style={draw=black} 
    ]
    \addplot[mark=none, color={black!10}, line width=1.0] table [col sep=comma, header=true, x expr=\coordindex, y index=1] {Plots/1D_Allen_Cahn_alpha_vary/Vary_alpha_0_nPts_201/Error.csv};
    \addplot[mark=none, color={black!30}, line width=1.0] table [col sep=comma, header=true, x expr=\coordindex, y index=1] {Plots/1D_Allen_Cahn_alpha_vary/Vary_alpha_2_nPts_201/Error.csv};
    \addplot[mark=none, color={black!50}, line width=1.0] table [col sep=comma, header=true, x expr=\coordindex, y index=1] {Plots/1D_Allen_Cahn_alpha_vary/Vary_alpha_4_nPts_201/Error.csv};
    \addplot[mark=none, color={black!70}, line width=1.0] table [col sep=comma, header=true, x expr=\coordindex, y index=1] {Plots/1D_Allen_Cahn_alpha_vary/Vary_alpha_6_nPts_201/Error.csv};
    \addplot[mark=none, color={black!90}, line width=1.0] table [col sep=comma, header=true, x expr=\coordindex, y index=1] {Plots/1D_Allen_Cahn_alpha_vary/Vary_alpha_8_nPts_201/Error.csv};
    \addplot[mark=none, color={black!110}, line width=1.0] table [col sep=comma, header=true, x expr=\coordindex, y index=1] {Plots/1D_Allen_Cahn_alpha_vary/Vary_alpha_10_nPts_201/Error.csv};
    \end{groupplot}
\end{tikzpicture}
     }{
  }
 \caption{Example \ref{Sec: Allen-Cahn equation}. Plots of the $\leb2$-error between $u_\theta, f_\theta$ and $u^*,f^*$, equation \eqref{eq: Allen-Cahn state and control exact}; for varying $\alpha$ values and $\epsilon =1$.}
    \label{fig:Allen-Cahn Sine1 L2 error}
\end{figure}
\begin{figure}
    \centering
    %\ifthenelse{\boolean{showTikz}}
        {
          \begin{tikzpicture}
    \begin{groupplot}[
        group style={
            group size=2 by 1,
            horizontal sep=2.5cm,
            vertical sep=2cm,
        },
        enlargelimits=false,
        axis line style={draw=black}
    ]
       \nextgroupplot[
        width=0.4\textwidth,
        ylabel= {$\Norm{ u_{\theta}-u^*}_{\leb{2}(\W)}$},
        xlabel={Update number $N_{\rm{Uz}}$},
        grid=both,
        ymode=log,
        log basis y={10},
        legend style={at={(1.05,1)}, anchor=north west, font=\scriptsize},
        axis line style={draw=black}
       ]
       \addlegendentry{$\epsilon = 0.2$}
       \addplot[mark=none, color={black!40}, line width=1.0] table [col sep=comma, header=true, x expr=\coordindex, y index=0] {Plots/1D_Allen_Cahn_eps_vary/Vary_alpha_4_nPts_201_eps_0.2/Error.csv};
       \addlegendentry{$\epsilon = 0.1$}
       \addplot[mark=none, color={black!70}, line width=1.0] table [col sep=comma, header=true, x expr=\coordindex, y index=0] {Plots/1D_Allen_Cahn_eps_vary/Vary_alpha_4_nPts_201_eps_0.1/Error.csv};
       \addlegendentry{$\epsilon = 0.05$}
       \addplot[mark=none, color={black!100}, line width=1.0] table [col sep=comma, header=true, x expr=\coordindex, y index=0] {Plots/1D_Allen_Cahn_eps_vary/Vary_alpha_4_nPts_201_eps_0.05/Error.csv};

      \nextgroupplot[
        width=0.4\textwidth,
        ylabel= {$\Norm{ f_{\theta}-f^*}_{\leb{2}(\W)}$},
        xlabel={Update number $N_{\rm{Uz}}$},
        grid=both,
        ymode=log,
        log basis y={10},
        legend style={at={(1.05,1)}, anchor=north west, font=\scriptsize},
        axis y line*=right,
        axis line style={draw=black}
      ]
      \addplot[mark=none, color={black!40}, line width=1.0] table [col sep=comma, header=true, x expr=\coordindex, y index=1] {Plots/1D_Allen_Cahn_eps_vary/Vary_alpha_4_nPts_201_eps_0.2/Error.csv};
      \addplot[mark=none, color={black!70}, line width=1.0] table [col sep=comma, header=true, x expr=\coordindex, y index=1] {Plots/1D_Allen_Cahn_eps_vary/Vary_alpha_4_nPts_201_eps_0.1/Error.csv};
      \addplot[mark=none, color={black!100}, line width=1.0] table [col sep=comma, header=true, x expr=\coordindex, y index=1] {Plots/1D_Allen_Cahn_eps_vary/Vary_alpha_4_nPts_201_eps_0.05/Error.csv};
    \end{groupplot}
\end{tikzpicture}  
     }{
  }
    \caption{Example \ref{Sec: Allen-Cahn equation}. Plots of the $\leb2$-error between $u_\theta$ and $u^*$, and $f_\theta$ and $f^*$; against the update number and $\epsilon \in [0.05,0.2]$, for $\alpha  = 10^{-4}$.}
    \label{fig:AllenCahnepsvary}
\end{figure}
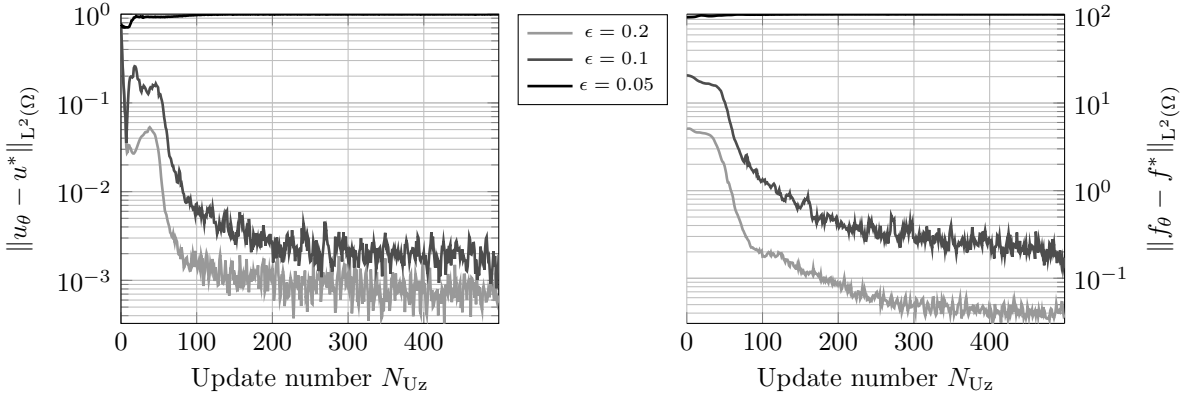
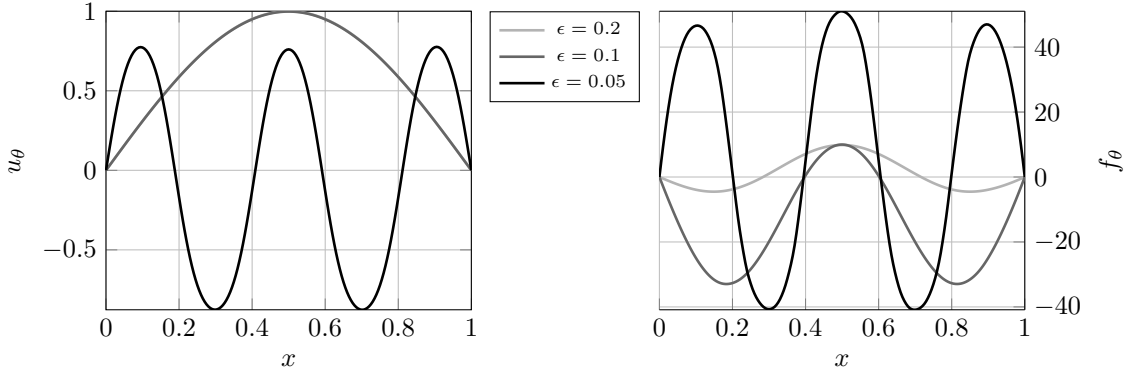
\begin{figure}[hbtp]
    \centering
    %\ifthenelse{\boolean{showTikz}}
        {
      \begin{tikzpicture}
    \begin{groupplot}[
        group style={
            group size=2 by 1,
            horizontal sep=2.5cm, 
            vertical sep=2cm,
        },
        enlargelimits=false,
        axis line style={draw=black} % Box around both plots
    ]
 
    % Left plot
      \nextgroupplot[
        width=0.39\textwidth,              
        xlabel={$x$},
        ylabel= {$u_{\theta}$},
        grid=both,
        xmin = 0,
        xmax = 1,
        %ymin = 1e-7,
        legend style={at={(1.05,1)}, anchor=north west, font=\scriptsize}, 
        axis line style={draw=black} % Box around the left plot
    ]
    % Manually specified plots
    
    \addlegendentry{$\epsilon = 0.2$}
    \addplot[mark=none, color={black!30}, line width=1.0] table [col sep=comma, header=true, x expr=\coordindex/200, y index=0] {Plots/1D_Allen_Cahn_eps_vary/Vary_alpha_4_nPts_201_eps_0.2/State.csv};
    \addlegendentry{$\epsilon = 0.1$}
    \addplot[mark=none, color={black!60}, line width=1.0] table [col sep=comma, header=true, x expr=\coordindex/200, y index=0] {Plots/1D_Allen_Cahn_eps_vary/Vary_alpha_4_nPts_201_eps_0.1/State.csv};
    \addlegendentry{$\epsilon = 0.05$}
    \addplot[mark=none, color={black!100}, line width=1.0] table [col sep=comma, header=true, x expr=\coordindex/200, y index=0] {Plots/1D_Allen_Cahn_eps_vary/Vary_alpha_4_nPts_201_eps_0.05/State.csv};

    % Left plot
      \nextgroupplot[
        width=0.39\textwidth,
        xlabel={$x$},
        ylabel= {$f_{\theta}$},
        grid=both,
        xmin = 0,
        xmax = 1,
        axis y line*=right,
        legend style={at={(1.05,1)}, anchor=north west, font=\scriptsize}, % Smaller legend
        axis line style={draw=black} % Box around the left plot
    ]
    % Manually specified plots
    
    \addplot[mark=none, color={black!30}, line width=1.0] table [col sep=comma, header=true, x expr=\coordindex/200, y index=0] {Plots/1D_Allen_Cahn_eps_vary/Vary_alpha_4_nPts_201_eps_0.2/Control.csv};

    \addplot[mark=none, color={black!60}, line width=1.0] table [col sep=comma, header=true, x expr=\coordindex/200, y index=0] {Plots/1D_Allen_Cahn_eps_vary/Vary_alpha_4_nPts_201_eps_0.1/Control.csv};

    \addplot[mark=none, color={black!100}, line width=1.0] table [col sep=comma, header=true, x expr=\coordindex/200, y index=0] {Plots/1D_Allen_Cahn_eps_vary/Vary_alpha_4_nPts_201_eps_0.05/Control.csv};
    \end{groupplot}
\end{tikzpicture}
    }{
  }
    \caption{Example \ref{Sec: Allen-Cahn equation}. Plots of the neural network approximate state and control, for $\alpha = 10^{-4}$ and varying $\epsilon$ values.}
    \label{fig:Allen-Cahn Sine1 state and control}
\end{figure}
\subsection{Example: Step function target}\label{sec: Step Allen Cahn}

Let target function $\mathcal{D}$ be given by
\begin{equation}\label{eq: Allen-Cahn step target}
    \mathcal{D}(x) := \funmultidefn{-1 & x \in \bra{0,\frac{1}{3}}\cup \bra{\frac{2}{3},1} 
    \\
    1 & x \in \bra{\frac{1}{3},\frac{2}{3}}
    \\
    0 & \text{ otherwise.}
    }
\end{equation}
In the previous examples, we constructed the target such that an
explicit solution was given; which we do not have in this example. In
Figure \ref{fig:Step_vary_alpha_State_Control}, we observe a change in
behaviour of the neural network solution as $\alpha$ decreases. The step discontinuity in the state results in a set of large values
in the control

The
state converges, with respect to update number, to the target $\mathcal{D}$
both pointwise and in $\leb2(\W)$, as seen in Figure
\ref{fig:Step_vary_alpha_Loss}. Additionally, we observe that the $\leb2$-norm of
$f_{\theta}$ increases proportionally to $\alpha^{-1/2}$.

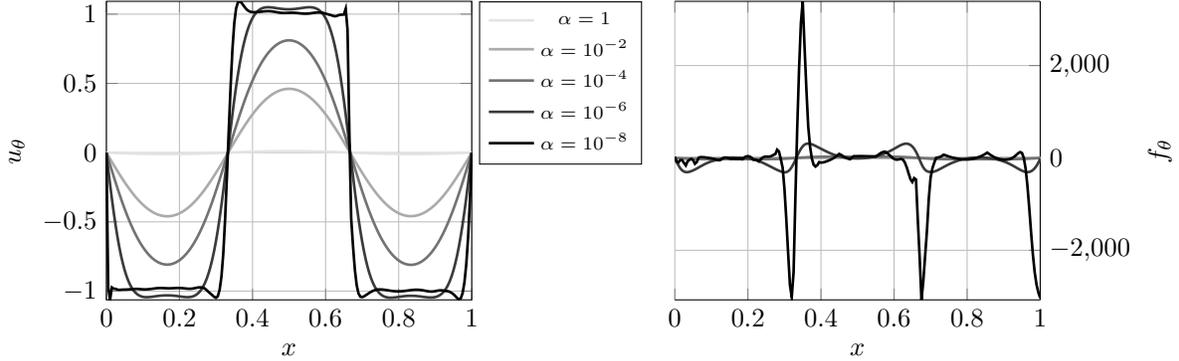
\begin{figure}[hbtp]
    \centering
    %\ifthenelse{\boolean{showTikz}}
        {
          \begin{tikzpicture}
    \begin{groupplot}[
        group style={
            group size=2 by 1,
            horizontal sep=2.7cm,
            vertical sep=2cm,
        },
        enlargelimits=false,
        axis line style={draw=black}
    ]
      \nextgroupplot[
        width=0.39\textwidth,
        xlabel={$x$},
        ylabel= {$u_{\theta}$},
        grid=both,
        legend style={at={(1.02,1)}, anchor=north west, font=\scriptsize},
        axis line style={draw=black}
    ]
    \addlegendentry{$\alpha=1$}
    \addplot[mark=none, color={black!11}, line width=1.0] table [col sep=comma, header=true, x expr=\coordindex/200, y index=0] {Plots/1D_Allen_Cahn_Step_fun_alpha_vary/Vary_alpha_0_eps_0.1/State.csv};
    \addlegendentry{$\alpha=10^{-2}$}
    \addplot[mark=none, color={black!33}, line width=1.0] table [col sep=comma, header=true, x expr=\coordindex/200, y index=0] {Plots/1D_Allen_Cahn_Step_fun_alpha_vary/Vary_alpha_2_eps_0.1/State.csv};
    \addlegendentry{$\alpha=10^{-4}$}
    \addplot[mark=none, color={black!55}, line width=1.0] table [col sep=comma, header=true, x expr=\coordindex/200, y index=0] {Plots/1D_Allen_Cahn_Step_fun_alpha_vary/Vary_alpha_4_eps_0.1/State.csv};
    \addlegendentry{$\alpha=10^{-6}$}
    \addplot[mark=none, color={black!77}, line width=1.0] table [col sep=comma, header=true, x expr=\coordindex/200, y index=0] {Plots/1D_Allen_Cahn_Step_fun_alpha_vary/Vary_alpha_6_eps_0.1/State.csv};
    \addlegendentry{$\alpha=10^{-8}$}
    \addplot[mark=none, color={black!100}, line width=1.0] table [col sep=comma, header=true, x expr=\coordindex/200, y index=0] {Plots/1D_Allen_Cahn_Step_fun_alpha_vary/Vary_alpha_8_eps_0.1/State.csv};
      \nextgroupplot[
        width=0.39\textwidth,
        ylabel= {$f_{\theta}$},
        xlabel={$x$},
        axis y line*=right,
        grid=both,
        legend style={at={(1.02,1)}, anchor=north west, font=\scriptsize},
        axis line style={draw=black}
    ]
    \addplot[mark=none, color={black!11}, line width=1.0] table [col sep=comma, header=true, x expr=\coordindex/200, y index=0] {Plots/1D_Allen_Cahn_Step_fun_alpha_vary/Vary_alpha_0_eps_0.1/Control.csv};
    \addplot[mark=none, color={black!33}, line width=1.0] table [col sep=comma, header=true, x expr=\coordindex/200, y index=0] {Plots/1D_Allen_Cahn_Step_fun_alpha_vary/Vary_alpha_2_eps_0.1/Control.csv};
    \addplot[mark=none, color={black!55}, line width=1.0] table [col sep=comma, header=true, x expr=\coordindex/200, y index=0] {Plots/1D_Allen_Cahn_Step_fun_alpha_vary/Vary_alpha_4_eps_0.1/Control.csv};
    \addplot[mark=none, color={black!77}, line width=1.0] table [col sep=comma, header=true, x expr=\coordindex/200, y index=0] {Plots/1D_Allen_Cahn_Step_fun_alpha_vary/Vary_alpha_6_eps_0.1/Control.csv};
    \addplot[mark=none, color={black!100}, line width=1.0] table [col sep=comma, header=true, x expr=\coordindex/200, y index=0] {Plots/1D_Allen_Cahn_Step_fun_alpha_vary/Vary_alpha_8_eps_0.1/Control.csv};
    \end{groupplot}
\end{tikzpicture}}{}
    \caption{Example \ref{sec: Step Allen Cahn}. Plots of the neural network approximation for the state and control, for $\alpha \in [10^{-8},1]$ and $\epsilon = 10^{-1}$.}
    \label{fig:Step_vary_alpha_State_Control}
\end{figure}

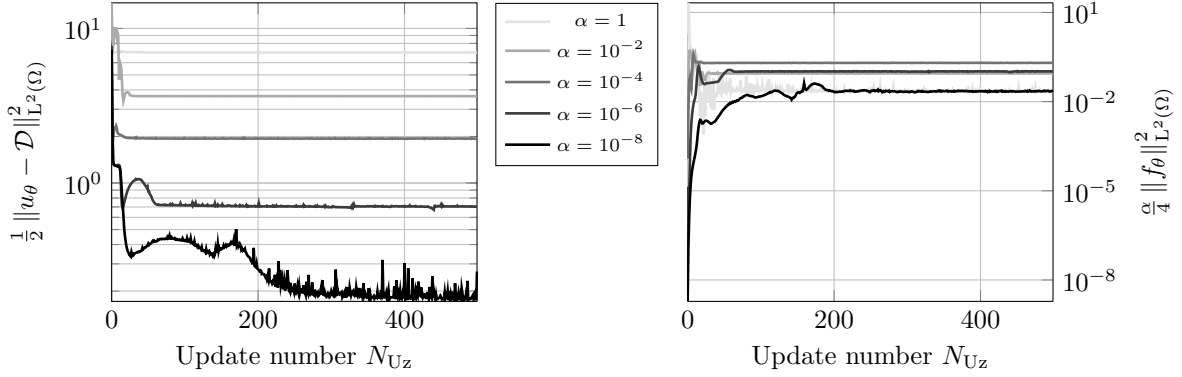
\begin{figure}[hbtp]
    \centering
    %\ifthenelse{\boolean{showTikz}}
               {\begin{tikzpicture}
    \begin{groupplot}[
        group style={
            group size=2 by 1,
            horizontal sep=2.8cm, 
            vertical sep=2cm,
        },
        enlargelimits=false,
        axis line style={draw=black} 
    ]
      \nextgroupplot[
        width=0.39\textwidth,
        xlabel={Update number $N_{\rm{Uz}}$},
        ylabel= {$\frac{1}{2}\Norm{u_{\theta}-\mathcal{D}}^2_{\leb2(\W)}$},
        ymode=log,
        log basis y={10},
        grid=both,
        legend style={at={(1.05,1)}, anchor=north west, font=\scriptsize}, 
        axis line style={draw=black} 
    ]
    \addlegendentry{$\alpha=1$}
    \addplot[mark=none, color={black!11}, line width=1.0] table [col sep=comma, header=true, x expr=\coordindex, y index=1] {Plots/1D_Allen_Cahn_Step_fun_alpha_vary/Vary_alpha_0_eps_0.1/Loss.csv};
    \addlegendentry{$\alpha=10^{-2}$}
    \addplot[mark=none, color={black!33}, line width=1.0] table [col sep=comma, header=true, x expr=\coordindex, y index=1] {Plots/1D_Allen_Cahn_Step_fun_alpha_vary/Vary_alpha_2_eps_0.1/Loss.csv};
    \addlegendentry{$\alpha=10^{-4}$}
    \addplot[mark=none, color={black!55}, line width=1.0] table [col sep=comma, header=true, x expr=\coordindex, y index=1] {Plots/1D_Allen_Cahn_Step_fun_alpha_vary/Vary_alpha_4_eps_0.1/Loss.csv};
    \addlegendentry{$\alpha=10^{-6}$}
    \addplot[mark=none, color={black!77}, line width=1.0] table [col sep=comma, header=true, x expr=\coordindex, y index=1] {Plots/1D_Allen_Cahn_Step_fun_alpha_vary/Vary_alpha_6_eps_0.1/Loss.csv};
    \addlegendentry{$\alpha=10^{-8}$}
    \addplot[mark=none, color={black!100}, line width=1.0] table [col sep=comma, header=true, x expr=\coordindex, y index=1] {Plots/1D_Allen_Cahn_Step_fun_alpha_vary/Vary_alpha_8_eps_0.1/Loss.csv};
    \nextgroupplot[
        width=0.39\textwidth,
        xlabel={Update number $N_{\rm{Uz}}$},
        ylabel= {$\frac{\alpha}{4}\Norm{f_{\theta}}^2_{\leb2(\W)}$},
        ymode=log,
        log basis y={10},
        grid=both,     
        axis y line*=right,
        legend style={at={(1.05,1)}, anchor=north west, font=\scriptsize}, 
        axis line style={draw=black} 
    ]
    \addplot[mark=none, color={black!11}, line width=1.0] table [col sep=comma, header=true, x expr=\coordindex, y index=3] {Plots/1D_Allen_Cahn_Step_fun_alpha_vary/Vary_alpha_0_eps_0.1/Loss.csv};
    \addplot[mark=none, color={black!33}, line width=1.0] table [col sep=comma, header=true, x expr=\coordindex, y index=3] {Plots/1D_Allen_Cahn_Step_fun_alpha_vary/Vary_alpha_2_eps_0.1/Loss.csv};
    \addplot[mark=none, color={black!55}, line width=1.0] table [col sep=comma, header=true, x expr=\coordindex, y index=3] {Plots/1D_Allen_Cahn_Step_fun_alpha_vary/Vary_alpha_4_eps_0.1/Loss.csv};
    \addplot[mark=none, color={black!77}, line width=1.0] table [col sep=comma, header=true, x expr=\coordindex, y index=3] {Plots/1D_Allen_Cahn_Step_fun_alpha_vary/Vary_alpha_6_eps_0.1/Loss.csv};
    \addplot[mark=none, color={black!100}, line width=1.0] table [col sep=comma, header=true, x expr=\coordindex, y index=3] {Plots/1D_Allen_Cahn_Step_fun_alpha_vary/Vary_alpha_8_eps_0.1/Loss.csv};
    \end{groupplot}
\end{tikzpicture}}{}
    \caption{Example \ref{sec: Step Allen Cahn}. Plots of the $\leb2$-norm of the difference between the state and the target function, and the control; for $\alpha \in [10^{-8},1]$ and $\epsilon = 10^{-1}$.}
    \label{fig:Step_vary_alpha_Loss}
\end{figure}

For a fixed $\alpha$, in Figure \ref{fig:Step_vary_eps_State_Control}
we observe a transition in behaviour of the neural network
approximation to the control at $\epsilon = 0.1$; where the peaks and
troughs flatten. A second transition in behaviour occurs when
$\epsilon = 0.05$, where the amplitude of the state decreases and the
frequency of both the state and control increases. Similar to Example
\ref{Sec: Allen-Cahn equation}, this transition in behaviour is the
result of the lack of uniqueness in the constraint.

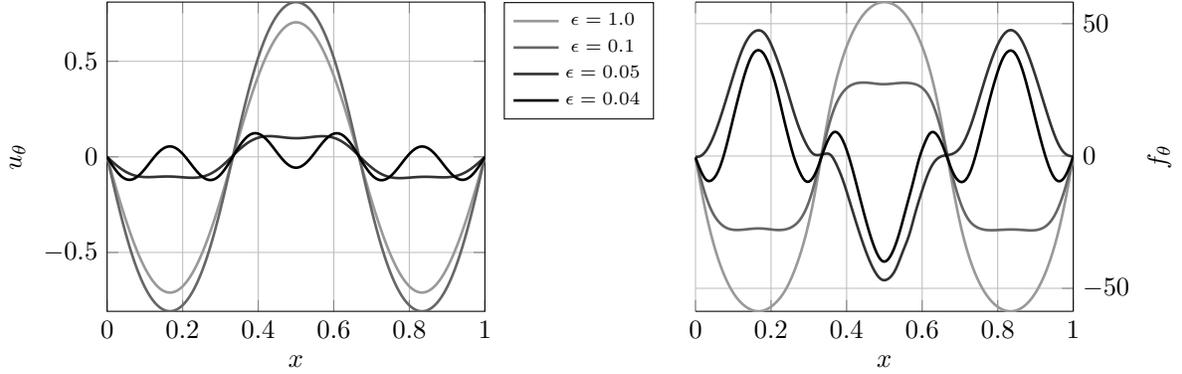
\begin{figure}[hbtp]
    \centering
    %\ifthenelse{\boolean{showTikz}}
        {
          \begin{tikzpicture}
    \begin{groupplot}[
        group style={
            group size=2 by 2,
            horizontal sep=2.8cm,
            vertical sep=2cm,
        },
        enlargelimits=false,
        axis line style={draw=black}
    ]
      \nextgroupplot[
        width=0.4\textwidth,
        xlabel={$x$},
        ylabel= {$u_{\theta}$},
        grid=both,
        legend style={at={(1.05,1)}, anchor=north west, font=\scriptsize}, 
        axis line style={draw=black}
    ]
    \addlegendentry{$\epsilon=1.0$}
    \addplot[mark=none, color={black!40}, line width=1.0] table [col sep=comma, header=true, x expr=\coordindex/200, y index=0] {Plots/1D_Allen_Cahn_Step_fun_eps_vary/Vary_alpha_4_eps_1.0/State.csv};
    \addlegendentry{$\epsilon=0.1$}
    \addplot[mark=none, color={black!60}, line width=1.0] table [col sep=comma, header=true, x expr=\coordindex/200, y index=0] {Plots/1D_Allen_Cahn_Step_fun_eps_vary/Vary_alpha_4_eps_0.1/State.csv};
    \addlegendentry{$\epsilon=0.05$}
    \addplot[mark=none, color={black!80}, line width=1.0] table [col sep=comma, header=true, x expr=\coordindex/200, y index=0] {Plots/1D_Allen_Cahn_Step_fun_eps_vary/Vary_alpha_4_eps_0.05/State.csv};
    \addlegendentry{$\epsilon=0.04$}
    \addplot[mark=none, color={black!100}, line width=1.0] table [col sep=comma, header=true, x expr=\coordindex/200, y index=0] {Plots/1D_Allen_Cahn_Step_fun_eps_vary/Vary_alpha_4_eps_0.04/State.csv};
      \nextgroupplot[
        width=0.4\textwidth,
        ylabel= {$f_{\theta}$},
        xlabel={$x$},
        axis y line*=right,
        grid=both,
        legend style={at={(1.05,1)}, anchor=north west, font=\scriptsize}, 
        axis line style={draw=black}
    ]
    \addplot[mark=none, color={black!40}, line width=1.0] table [col sep=comma, header=true, x expr=\coordindex/200, y index=0] {Plots/1D_Allen_Cahn_Step_fun_eps_vary/Vary_alpha_4_eps_1.0/Control.csv};
    \addplot[mark=none, color={black!60}, line width=1.0] table [col sep=comma, header=true, x expr=\coordindex/200, y index=0] {Plots/1D_Allen_Cahn_Step_fun_eps_vary/Vary_alpha_4_eps_0.1/Control.csv};
    \addplot[mark=none, color={black!80}, line width=1.0] table [col sep=comma, header=true, x expr=\coordindex/200, y index=0] {Plots/1D_Allen_Cahn_Step_fun_eps_vary/Vary_alpha_4_eps_0.05/Control.csv};
    \addplot[mark=none, color={black!100}, line width=1.0] table [col sep=comma, header=true, x expr=\coordindex/200, y index=0] {Plots/1D_Allen_Cahn_Step_fun_eps_vary/Vary_alpha_4_eps_0.04/Control.csv};    
    \end{groupplot}
\end{tikzpicture}}{}
    \caption{Example \ref{sec: Step Allen Cahn}. Plots of the neural network approximation for the state and control, for $\epsilon \in [0.04,1]$ and $\alpha = 10^{-4}$.}
    \label{fig:Step_vary_eps_State_Control}
\end{figure}

\subsection{Example: 2d Allen-Cahn equation}\label{sec: Allen-KAHN}

In the following we consider the two-dimensional Allen-Cahn \eqref{eq:
  Allen-Cahn equation} as the PDE constraint. The two-dimensional
target function $\mathcal{D}$ is given by the normalised greyscale
values of the image in Figure \ref{fig:Mickey}. In Figures
\ref{fig:Kahn_alpha_6_eps_0.1}-\ref{fig:Kahn_alpha_8_eps_0.1}, plots
of the solution $u_{\theta}$ are illustrated. 

\begin{figure}[htbp]
    \centering
    \begin{subfigure}{0.48\textwidth} 
        \centering    \includegraphics[width=\textwidth]{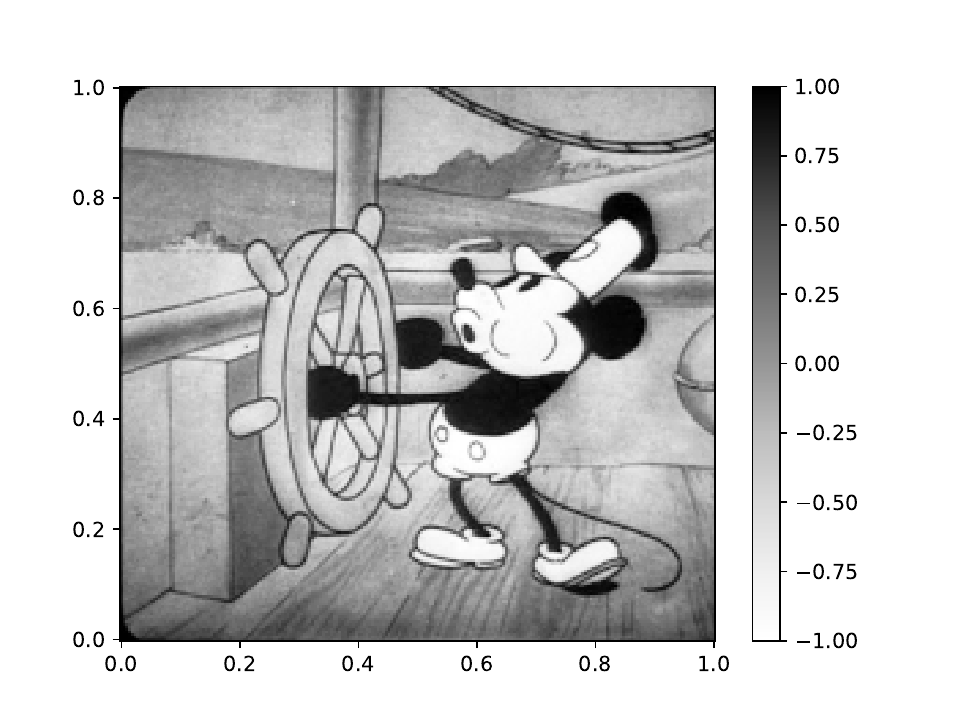}
        \caption{Plot of $\mathcal{D}$.}
    \label{fig:Mickey}
    \end{subfigure}
    \\
    \begin{subfigure}{0.48\textwidth}  
        \includegraphics[width=\textwidth]{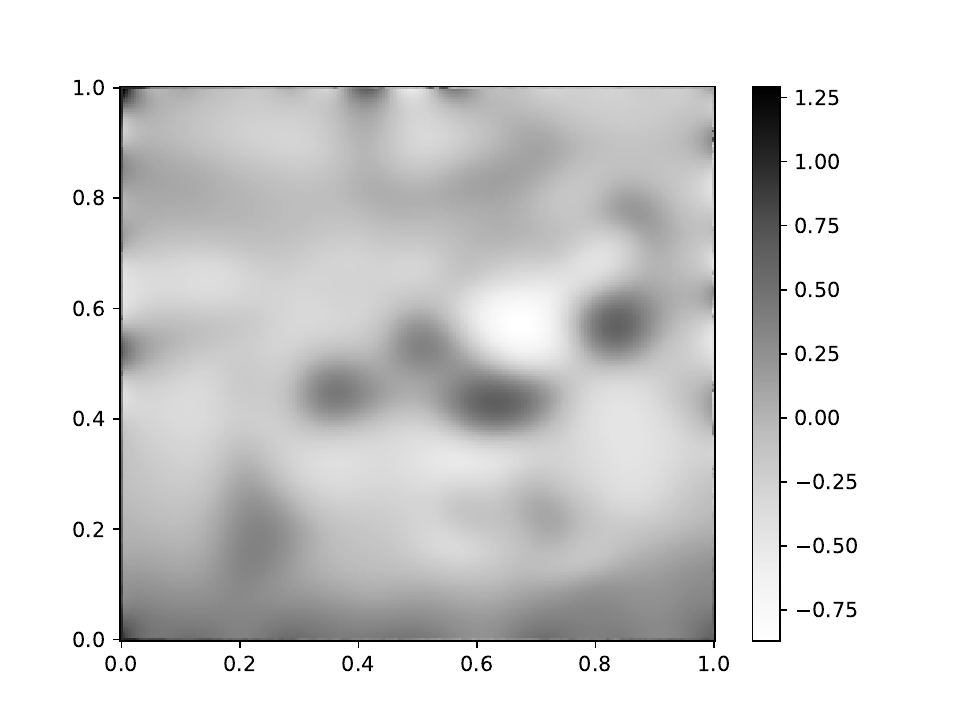}
    \caption{Plot of $u_\theta$, $\alpha = 10^{-6}, \epsilon = 0.1$.}
    \label{fig:Kahn_alpha_6_eps_0.1}
    \end{subfigure}
    \hfill
    \begin{subfigure}{0.48\textwidth}   
        \includegraphics[width=\textwidth]{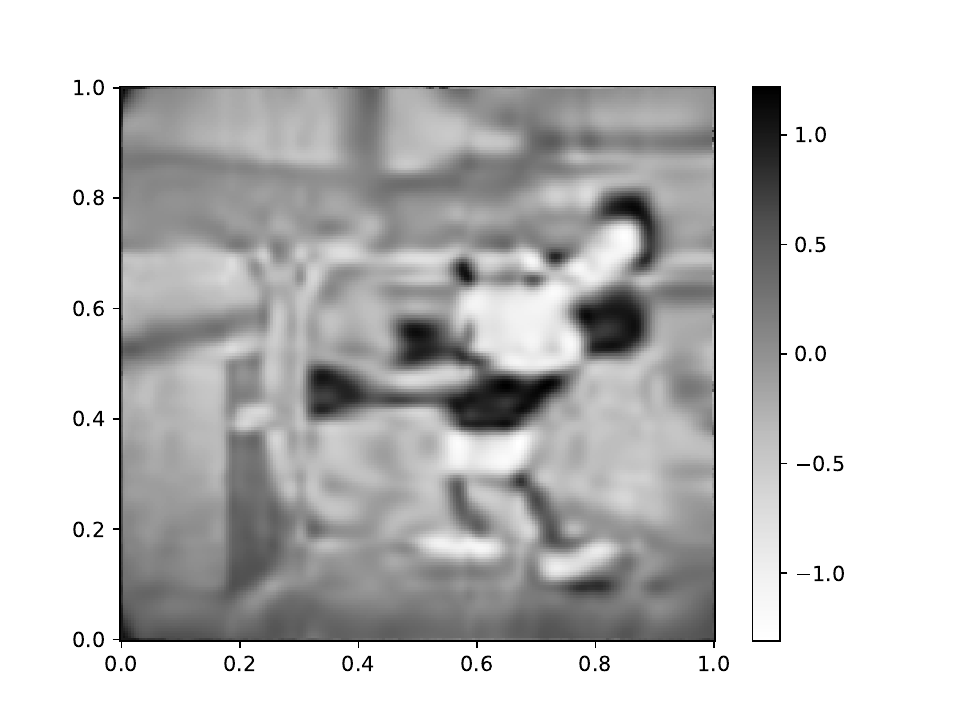}
    \caption{Plot of $u_\theta$, $\alpha = 10^{-10}, \epsilon = 0.1$.}
    \label{fig:Kahn_alpha_8_eps_0.1}
    \end{subfigure}
    \caption{(A) The two-dimensional target function
      $\mathcal{D}$. The still is of the public domain character
      Mickey Mouse, from Steamboat Willie (1928). (B)-(C) Plots of the
      state $u_{\theta}$ for varying values of $\alpha$ and
      $\epsilon = 0.1$.}
\end{figure}

\section{Proof of Theorems \ref{the:convergence} and \ref{the:box_constraint_converge}}
\label{sec:proof}

\begin{lemma}\label{lemma:upperbound1}
  Let $u^*, f^*, z^* \in \Hspace(\W) \times \leb2(\W) \times \leb{2}_0$
  denote the saddle point of $L$ \eqref{eq:u*,f*,z* defn}. Let the
  sequences $\setbra{u^k}_{k=0}^\infty \subset \Hspace(\W)$,
  $\setbra{f^k}_{k=0}^\infty \subset \leb2(\W)$ and
  $\setbra{z^k}_{k=0}^\infty \subset \leb{2}_0(\W)$ be generated
  through the iterative method \eqref{eq:Uzawa scheme Poisson}. Then
  we have
  \begin{equation}
    \ip{\Delta (u^*-u^k) + f^*-f^*}{z^* - z^k}
    =
    \Norm{u^*-u^k}^2_{\leb2(\W)}
    +
    \dfrac{\alpha}{2}\Norm{f^*-f^k}^2_{\leb2(\W)}
    +
    \dfrac{\alpha}{2}\Norm{\Delta(u^*-u^k)}^2_{\leb2(\W)}.
  \end{equation}
\end{lemma}

\begin{proof}
  Integrating by parts and using the boundary conditions on $u$ and
  $z$ imply that
    \begin{equation}\label{eq: lemma 5.1 part 1}
      \ip{\Delta (u^*-u^k)}{z^* - z^k} + \ip{f^*-f^k}{z^* - z^k} = \braket{u^*-u^k}{\Delta(z^* - z^k)}_{\Hspace \times \sobh{*}} + \ip{f^*-f^k}{z^* - z^k}.
    \end{equation}
    As $u^k, f^k$ and $u^*, f^*$ are critical points of the Lagrangian
    $L$, the first order optimality conditions are
    \begin{equation}
        \begin{split}
            \ip{u^* - \mathcal{D}}{\phi} 
            + \dfrac{\alpha}{2}\ip{\Delta u^*}{\Delta \phi}
            - \braket{\phi}{\Delta z^*}_{\Hspace \times \sobh{*}} &=0 
            \quad \forall \phi \in \Hspace(\W), 
            \\ 
            \dfrac{\alpha}{2}\ip{f^*}{\psi}
            - \ip{z^*}{\psi} &=0 
            \quad \forall \psi \in \leb2(\W),
            \\
            \ip{u^k - \mathcal{D}}{\phi} 
            + \dfrac{\alpha}{2}\ip{\Delta u^k}{\Delta \phi}
            - \braket{\phi}{\Delta z^k}_{\Hspace \times \sobh{*}} &=0 
            \quad 
            \forall \phi \in \Hspace(\W), 
            \\
            \dfrac{\alpha}{2}\ip{f^k}{\psi}
            - \ip{z^k}{\psi} &=0 \quad \forall \psi \in \leb2(\W).
        \end{split}
    \end{equation}
        
    Let $\phi := u^*-u^k$ and $\psi := f^*-f^k$, and consider the difference
    \begin{equation}\label{eq: lemma 5.1 part 2}
        \begin{split}
            \Norm{u^k-u^*}^2_{\leb2(\W)} +\dfrac{\alpha}{2}\Norm{\Delta(u^k-u^*)}^2_{\leb2(\W)} - \braket{u^*-u^k}{\Delta(z^*-z^k)}_{\Hspace \times \sobh{*}} &=0 
            \\ 
            \dfrac{\alpha}{2}\Norm{f^k-f^*}^2_{\leb2(\W)} - \ip{f^*-f^k}{z^*-z^k} &=0.
            \end{split}
    \end{equation}
    Substituting equation \eqref{eq: lemma 5.1 part 2} into equation
    \eqref{eq: lemma 5.1 part 1} we obtain our result.
\end{proof}
\begin{lemma}\label{lemma:upperbound2}
  Let $L:\Hspace(\W) \times \leb2(\W) \times \leb2_0(\W) \rightarrow
  \reals$ denote the Lagrangian \eqref{eq:Lagrangian}. Let $u^*,f^*,
  z^* \in \Hspace(\W) \times \leb2(\W) \times \leb2_0(\W)$ denote the
  saddle point of $L$, equation \eqref{eq:u*,f*,z* defn}. Then
    \begin{equation}
        z^* = z^* + \rho \bra{\Delta u^* + f^*}.
    \end{equation}
\end{lemma}
\begin{proof}
  Note that $z^*$ maximises the Lagrangian, thus
  \begin{equation}
    L(u,f,z)  \leq L(u,f,z^*) \qquad \forall u,f,z \in \Hspace(\W) \times \leb2(\W) \times \leb2_0(\W).
  \end{equation}
  Substituting the definition of the Lagrangian Now, in view of
  \eqref{eq:Lagrangian} we have
  \begin{equation}
    {J}(u,f) + \ip{\Delta u + f}{z}  \leq {J}(u,f) + \ip{\Delta u + f}{z^*}  \qquad \forall u,f,z \in \Hspace(\W) \times \leb2(\W) \times \leb2_0(\W).
  \end{equation}
  Rearranging we obtain
  \begin{equation}
    \ip{\Delta u + f}{z^* - z} \geq 0 \qquad \forall u,f,z \in \Hspace(\W) \times \leb2(\W) \times \leb2_0.
  \end{equation}
  Thus for $u = u^*$ and $f = f^*$ we have that
  \begin{equation}
    \begin{split}
      0  \leq& \ip{z^* - z^*}{z^* - z} + \rho \ip{\Delta u^* + f^*}{z^* - z}
      \\ =&
      \ip{z^* + \rho\bra{\Delta u^* + f^*}}{z^* - z}-\ip{z^*}{z^* - z} \qquad \forall z \in \leb2_0(\W),
    \end{split}
  \end{equation}
  as required.
\end{proof}

\begin{lemma}[Elliptic Regularity {\cite[Theorem 4.2.2.2]{grisvard2011elliptic}}]
  \label{lemma: Elliptic Regularity}
  Let $\W \subset \real^d$ be a bounded, convex polygonal domain. For
  each $f \in \leb2(\W)$, there exists a unique $u\in\Hspace(\W)$ such
  that
  \begin{equation}
    \begin{split}
      -\Delta u &= f \text{ in }\W,
      \\
      u &= 0 \text{ on } \partial \W.
    \end{split}
  \end{equation}
  Moreover, there exists a constant $C_{\rm{reg}} >0$ such that
  \begin{equation}
    \Norm{u}_{\sobh2(\W)}\leq C_{\rm{reg}}\Norm{f}_{\leb2(\W)}.
  \end{equation}
\end{lemma}

\subsection{Proof of Theorem \ref{the:convergence}}\label{sec: Proof of the convergence}

The proof of this result takes inspiration from
\cite{CiarletMiaraThomas:1989}. Using the definition of the Uzawa
update \eqref{eq:Uzawa scheme Poisson} and Lemma
\ref{lemma:upperbound2} we have that
\begin{equation}
  \begin{split}
    \Norm{z^{k+1}-z^*}^2_{\leb2(\W)}
    &=
    \Norm{\bra{z^k - \rho(\Delta u^k + f^k)}-\bra{z^* - \rho(\Delta u^* + f^*)}}^2_{\leb2(\W)},\\
    &= \Norm{\bra{z^k-z^*} - \rho \bra{\Delta\bra{u^k-u^*} +\bra{f^k-f^*}}}^2_{\leb2(\W)}.
  \end{split}
\end{equation}
Now by the definition of the norm and Lemma \ref{lemma:upperbound1} 
\begin{equation}
    \begin{split}
      \Norm{z^{k+1}-z^*}^2_{\leb2(\W)}  
      &=
      \Norm{z^k-z^*}^2_{\leb2(\W)} + \rho^2\Norm{\Delta\bra{u^k-u^*} +\bra{f^k-f^*}}^2_{\leb2(\W)}
      \\
      &\qquad - 2\rho \ip{z^k-z^*}{\Delta\bra{u^k-u^*} +\bra{f^k-f^*}},
      \\
      &=\Norm{z^k-z^*}^2_{\leb2(\W)} + \rho^2\Norm{\Delta\bra{u^k-u^*} +\bra{f^k-f^*}}^2_{\leb2(\W)} \\
      &\qquad - 2\rho  \Norm{u^*-u^k}^2_{\leb2(\W)} - \alpha \rho \Norm{f^*-f^k}^2_{\leb2(\W)}-\alpha \rho \Norm{\Delta(u^*-u^k)}^2_{\leb2(\W)}.
    \end{split}
  \end{equation}
  The triangle and Young's inequalities imply
  \begin{equation}\label{eq:zk+1,zk inequality}
    \begin{split}
      \Norm{z^{k+1}-z^*}^2_{\leb2(\W)}
      \leq
      &\Norm{z^k-z^*}^2_{\leb2(\W)}
      +
      \bra{2\rho^2-\alpha \rho}\bra{\Norm{\Delta\bra{u^k-u^*}}^2_{\leb2(\W)}
        +
        \Norm{\bra{f^k-f^*}}^2_{\leb2(\W)} } .
    \end{split}
  \end{equation}
  If $u^k = u^*$, $f^k =
  f^*$ then Theorem \ref{the:convergence} holds trivially. Thus without loss of generality assume that $u^k \neq u^*$, $f^k \neq
  f^*$ for all $k \geq 0$. We have that as $\rho \leq \tfrac \alpha 2$
    \begin{equation}
      \bra{2\rho^2-\alpha \rho}\bra{\Norm{\Delta\bra{u^k-u^*}}^2_{\leb2(\W)} + \Norm{f^k-f^*}^2_{\leb2(\W)} }  < 0,
    \end{equation}
    which implies that
    \begin{equation}
      \Norm{z^{k+1}-z^*}^2_{\leb2(\W)}
      <
      \Norm{z^k-z^*}^2_{\leb2(\W)}.
    \end{equation}    
    Hence
    \begin{equation}
      \lim_{k\to \infty}
      \Norm{z^{k}-z^*}^2_{\leb2(\W)}
      -
      \Norm{z^{k+1}-z^*}^2_{\leb2(\W)} = 0.
    \end{equation}
    Hence, rearranging \eqref{eq:zk+1,zk inequality}, we obtain
    \begin{equation}
      \begin{split}
        \bra{\alpha \rho - 2\rho^2}\bra{\Norm{\Delta\bra{u^k-u^*}}^2_{\leb2(\W)} + \Norm{f^k-f^*}^2_{\leb2(\W)} } 
            \leq &\Norm{z^k-z^*}^2_{\leb2(\W)} - \Norm{z^{k+1}-z^*}^2_{\leb2(\W)}.
        \end{split}
    \end{equation}
    Thus applying Lemma \ref{lemma: Elliptic Regularity} we have that
    \begin{equation}
      \begin{split}
        0 \leq \dfrac{1}{C_{reg}^2} \Norm{u^k-u^*}^2_{\sobh2(\W)} + \Norm{f^k-f^*}^2_{\leb2(\W)} 
        \leq &\dfrac{\Norm{z^k-z^*}^2_{\leb2(\W)} - \Norm{z^{k+1}-z^*}^2_{\leb2(\W)}}{\alpha \rho - 2\rho^2}.
        \end{split}
    \end{equation}
    Whence
    \begin{equation}
      \lim_{k\to \infty} \dfrac{1}{C_{reg}^2} \Norm{u^k-u^*}^2_{\sobh2(\W)} + \Norm{f^k-f^*}^2_{\leb2(\W)} = 0,
    \end{equation}
    as required.

\begin{lemma}\label{lemma:upperbound1box}
  Recall the function spaces $\Hspace(\W, \reals^+) \subset \Hspace(\W)$ and
  $\leb{2}_0(\W, \reals^+) \subset \leb2_0(\W)$, equation
  \eqref{eq:defn_constrained_functionspace}. Let $L:\Hspace(\W, \reals^+) \times
  \leb2(\W,\reals^+) \times \leb{2}_0(\W, \reals^+) \rightarrow \reals$ be the
  restricted Lagrangian \eqref{eq:Lagrangian}. Let $u^*, f^*, z^*
  \in \Hspace(\W, \reals^+) \times \leb2(\W,\reals^+) \times \leb{2}_0(\W, \reals^+)$ denote
  the saddle point of $L$ \eqref{eq:u*,f*,z* defn}. Let the sequences
  $\setbra{u^k}_{k=0}^\infty \subset \Hspace(\W, \reals^+)$,
  $\setbra{f^k}_{k=0}^\infty \subset \leb2(\W, \real^+)$ and
  $\setbra{z^k}_{k=0}^\infty \subset \leb{2}_0(\W, \reals^+)$ be generated through
  the iterative method \eqref{eq:Uzawa scheme Poisson}. Then we have
    \begin{equation}
        \begin{split}
            \Norm{u^k-u^*}^2_{\leb{2}(\W)}
        + \dfrac{\alpha}{2}\Norm{f^k-f^*}_{\leb{2}(\W)}^2
        + \dfrac{\alpha}{2}\Norm{\Delta (u^k-u^*)}_{\leb{2}(\W)}^2
        \\
        \leq 
        \braket{u^*-u^k}{\Delta(z^*-z^k)}_{\Hspace\times \sobh{*}} +\ip{u^*-u^k}{f^*-f^k}.
        \end{split}
    \end{equation} 
\end{lemma}
\begin{proof}
    We may use integration by parts and applying the enforced
    Dirichlet boundary conditions on $\Hspace(\W, \reals^+)$ and
    $\leb{2}_0(\W, \reals^+)$ to deduce that
    \begin{equation}
      \ip{\Delta (u^*-u^k)}{z^* - z^k} + \ip{f^*-f^k}{z^* - z^k}
      =
      \braket{u^*-u^k}{\Delta(z^* - z^k)}_{\Hspace\times \sobh{*}} + \ip{f^*-f^k}{z^* - z^k}.
    \end{equation}
    As $u^*, f^*, z^*$ and $u^k, f^k, z^k$ are saddle points of their
    respective systems, we have by the first order optimality
    conditions that
    \begin{equation}
        \begin{split}
          \ip{u^*-\mathcal{D}}{\phi_u - u^*}
          +
          \dfrac{\alpha}{2}\ip{\Delta u^*}{\Delta (\phi_u-u^*)}
          +
          \braket{\phi_u - u^*}{\Delta z^*}_{\Hspace \times \sobh{*}}
          +
          \dfrac{\alpha}{2}\ip{f^*}{\phi_f - f^*}
          \\
          +
          \ip{z^*}{\phi_f - f^*} \geq 0 \qquad \forall \phi_u, \phi_f \in \Hspace(\W, \reals^+) \times \leb2(\W,\real^+) 
        \end{split}
    \end{equation}
    \begin{equation}
        \begin{split}
        \ip{u^k-\mathcal{D}}{\psi_u - u^k} 
        + \dfrac{\alpha}{2}\ip{\Delta u^k}{\Delta (\psi_u-u^k)} 
        + \braket{\psi_u - u^k}{\Delta z^k}_{\Hspace \times \sobh{*}}
        +\dfrac{\alpha}{2}\ip{f^k}{\psi_f - f^k} \\
        + \ip{z^k}{\psi_f - f^k} \geq 0 \qquad \forall \psi_u, \psi_f \in \Hspace(\W,\real^+) \times \leb2(\W,\real^+) 
        \end{split}
    \end{equation}
    Thus setting $\phi_u = u^k$, $\phi_f = f^k$, $\psi_u = u^*$ and
    $\psi_f = f^*$, and considering the sum of the above equations, we
    have our result.
\end{proof}

\begin{lemma}\label{lemma:upperbound2box}
    Let $L:\Hspace(\W, \reals^+) \times \leb2(\W,\real^+) \times \leb{2}_0(\W, \reals^+)
    \rightarrow \reals$ denote the Lagrangian
    \eqref{eq:Lagrangian}. Let $u^*,f^*, z^* \in \Hspace(\W, \reals^+) \times
    \leb2(\W,\real^+) \times \leb{2}_0(\W, \reals^+)$ denote the saddle point of
    $L$, equation \eqref{eq:u*,f*,z* defn}. Let $P^+:
    \leb2(\W)\rightarrow \leb2(\W, \real^+)$ be the canonical
    projection operator, equation \eqref{eq:proj_map}. Then
    \begin{equation}
        z^* = P^{+}\bra{z^* + \rho \bra{\Delta u^* + f^*}}.
    \end{equation}
\end{lemma}
\begin{proof}
    The proof is similar to the proof of Lemma \ref{lemma:upperbound2}
\end{proof}
\begin{lemma}[Variational Elliptic Regularity {\cite[Theorem 3.2.3.1]{grisvard2011elliptic}}]
 \label{lemma:variational_elliptic_regularity}
  Let $\W \subset \real^d$ be a bounded, convex polygonal domain. For each $f
  \in \leb2(\W, \reals^+)$, there exists a unique $u\in\Hspace(\W, \reals^+)$ such that
  \begin{equation}
    \begin{split}
      \ip{f+\Delta u}{\phi -u} &\geq 0 \qquad \forall \phi \in \Hspace(\W, \reals^+).
    \end{split}
  \end{equation}
  Moreover, there exists a constant $C_{\rm{reg}} >0$ such that
  \begin{equation}
    \Norm{u}_{\sobh2(\W)}\leq C_{\rm{reg}}\Norm{f}_{\leb2(\W)}.
  \end{equation}
\end{lemma}

\subsection{Proof of Theorem \ref{the:box_constraint_converge}}
    As a consequence of the definition of the Uzawa update, from Equation \eqref{eq:box_constraint_iter}, and Lemma \ref{lemma:upperbound2box} we have that
    \begin{equation}
        \Norm{z^{k+1}-z^*}^2_{\leb2(\W)} = \Norm{P^+\bra{z^k - \rho(\Delta u^k + f^k)}-P^+\bra{z^* - \rho(\Delta u^* + f^*)}}^2_{\leb2(\W)}.
    \end{equation}
    The positive projection operator $P^+$ is convex thus
    \begin{equation}
        \begin{split}
            \Norm{z^{k+1}-z^*}^2_{\leb2(\W)} &\leq \Norm{P^+\bra{\bra{z^k - \rho(\Delta u^k + f^k)}-\bra{z^* - \rho(\Delta u^* + f^*)}}}^2_{\leb2(\W)}.
        \end{split}
    \end{equation}
    Additionally, the operator satisfies $\Norm{P^+ \phi}_{\leb2(\W)}\leq \Norm{\phi}_{\leb2(\W)}$, thus
    \begin{equation}
        \begin{split}
            \Norm{z^{k+1}-z^*}^2_{\leb2(\W)} &\leq \Norm{\bra{z^k - \rho(\Delta u^k + f^k)}-\bra{z^* - \rho(\Delta u^* + f^*)}}^2_{\leb2(\W)}.
        \end{split}
    \end{equation}
    Now by the definition of the norm and Lemma \ref{lemma:upperbound1box} %\tristan{put the right label in}
  \begin{equation}
    \begin{split}
      \Norm{z^{k+1}-z^*}^2_{\leb2(\W)}  
      &\leq \Norm{z^k-z^*}^2_{\leb2(\W)} + \rho^2\Norm{\Delta\bra{u^k-u^*} +\bra{f^k-f^*}}^2_{\leb2(\W)} \\
      &\qquad - 2\rho  \Norm{u^*-u^k}^2_{\leb2(\W)} - \alpha \rho \Norm{f^*-f^k}^2_{\leb2(\W)}-\alpha \rho \Norm{\Delta(u^*-u^k)}^2_{\leb2(\W)}.
    \end{split}
  \end{equation}
  Applying the same methodology as \S \ref{sec: Proof of the convergence}, we have
  \begin{equation}
      \lim_{k\to \infty}
      \Norm{z^{k}-z^*}^2_{\leb2(\W)}
      -
      \Norm{z^{k+1}-z^*}^2_{\leb2(\W)} = 0.
    \end{equation}
    Hence
    \begin{equation}
      \begin{split}
        \lim\limits_{k \rightarrow \infty }\Norm{\Delta(u^k-u^*)}^2_{\leb2(\W)} + \Norm{f^k-f^*}^2_{\leb2(\W)} 
        \leq & \lim\limits_{k \rightarrow \infty }\dfrac{\Norm{z^k-z^*}^2_{\leb2(\W)} - \Norm{z^{k+1}-z^*}^2_{\leb2(\W)}}{\alpha \rho - 2\rho^2}\rightarrow 0.
        \end{split}
    \end{equation}
    Thus applying Lemma \ref{lemma:variational_elliptic_regularity}, we have that
    \begin{equation}
      \begin{split}
        0 \leq \lim\limits_{k \rightarrow \infty }\dfrac{1}{C_{\text{reg}}^2}\Norm{u^k-u^*}^2_{\sobh2(\W)} + \Norm{f^k-f^*}^2_{\leb2(\W)} 
        \leq  0.
        \end{split}
    \end{equation}
    Thus we have our result.

    \section*{Acknowledgements}
    Tristan Pryer is grateful for the hospitality of ITE, Crete, where this work
    was initiated. Aaron Pim and Tristan Pryer received support from the EPSRC programme
    grant EP/W026899/1. Tristan Pryer was also supported by the Leverhulme
    RPG-2021-238 and EPSRC grant EP/X030067/1.

    \printbibliography
\end{document}